\documentclass[11pt]{amsart}
\title[Fractional Brauer configuration algebras]{Fractional Brauer configuration algebras I:\\ definitions and examples}
\author{Nengqun Li and Yuming Liu*}

\address{Nengqun Li
\newline School of Mathematics
\newline Liaoning Normal University
\newline Dalian 116029
\newline P.R.China}
\email{wd0843@163.com}

\address{Yuming Liu
\newline School of Mathematical Sciences
\newline Laboratory of Mathematics and Complex Systems
\newline Beijing Normal University
\newline Beijing 100875
\newline P.R.China}
\email{ymliu@bnu.edu.cn}

\date{version of \today}
\usepackage{amsmath}
\usepackage{amssymb}
\usepackage{amsxtra}
\usepackage{abstract}
\usepackage{mathrsfs}
\usepackage{mathtools}
\usepackage[all]{xy}
\usepackage{amscd}
\usepackage{amsthm}
\usepackage[dvips]{graphicx}
\usepackage{ulem}
\usepackage{color}
\usepackage{appendix}
\usepackage{tikz}

\newtheorem{Thm}{Theorem}[section]
\newtheorem{Lem}[Thm]{Lemma}
\newtheorem{Def}[Thm]{Definition}
\newtheorem{Cor}[Thm]{Corollary}
\newtheorem{Prop}[Thm]{Proposition}
\newtheorem{Ex1}[Thm]{Example}
\newtheorem{Rem1}[Thm]{Remark}

\newcommand{\lra}{\longrightarrow}

\newcommand{\ra}{\rightarrow}
\newcommand{\sdp}{\times\kern-.2em\vrule height1.1ex depth-.05ex}
\newcommand{\epi}{\lra \kern-.8em\ra}

 \normalbaselines
\begin{document}
\renewcommand{\thefootnote}{\alph{footnote}}
\setcounter{footnote}{-1} \footnote{* Corresponding author.}
\setcounter{footnote}{-1} 
\renewcommand{\thefootnote}{\alph{footnote}}
\setcounter{footnote}{-1} 

\maketitle

\begin{abstract}
In 2017, Green and Schroll introduced a generalization of Brauer graph algebras which they call Brauer configuration algebras. In the present paper, we further generalize Brauer configuration algebras to fractional Brauer configuration algebras by generalizing Brauer configurations to fractional Brauer configurations. The fractional Brauer configuration algebras are locally bounded but neither finite-dimensional nor symmetric in general. We show that if the fractional Brauer configuration is of type S (resp. of type MS), then the corresponding fractional Brauer configuration algebra is a locally bounded Frobenius algebra (resp. a locally bounded special multiserial Frobenius algebra). Moreover, we show that over an algebraically closed field, the class of finite-dimensional indecomposable representation-finite fractional Brauer configuration algebras in type S coincides with the class of basic indecomposable finite-dimensional standard representation-finite self-injective algebras.
\end{abstract}	

\section{Introduction}

In \cite{GS}, Green and Schroll introduced a generalization of Brauer graph
algebras which they call Brauer configuration algebras. As each Brauer graph algebra is defined by a Brauer graph, each Brauer configuration algebra is defined by a Brauer configuration. Both a Brauer graph and a Brauer configuration are given by some combinatorial data, each of which encodes the representation theory of the corresponding algebra.
It is known that over an algebraically closed field, the class of Brauer graph algebras coincides with the class of symmetric special biserial algebras (see for example \cite{S1}) and is closed under derived equivalence (see \cite{AZ}) (here and throughout when we say ``closed under derived equivalence", we always mean ``closed under derived equivalence up to Morita equivalence".); and the class of Brauer configuration algebras coincides with the class of symmetric special multiserial algebras (see \cite{GS2}).

\medskip
{\it In the present paper, we will give a further generalization of Brauer configurations which we call fractional Brauer configurations and define the corresponding fractional Brauer configuration categories and fractional Brauer configuration algebras.}

\medskip
Our motivation comes from two reasons. The first one is that unlike the Brauer graph algebras case, the class of Brauer configuration algebras is not closed under derived equivalence (see Example \ref{example-not-invariant-under-de}). The second one is related to the class of basic indecomposable finite-dimensional representation-finite self-injective algebras (abbr. RFS algebras). It is known that Brauer graph algebras of finite representation type are equal to Brauer tree algebras and form a subclass of RFS algebras of type $A_n$ (see \cite{S2,A}). We find that by allowing to take suitable fractional multiplicities in defining Brauer trees we can include naturally other subclass of RFS algebras of type $A_n$, however the algebras in such subclass are not symmetric any more.

Our first step is to define a fractional Brauer configuration. For simplicity, let us show how to define a fractional Brauer graph. A Brauer graph $\Gamma=(\Gamma_0,\Gamma_1,\mu,\mathfrak{o})$ is defined by a set $\Gamma_0$ of vertices (with a multiplicity function $\mu: \Gamma_0\rightarrow \mathbb{Z}_+$ on it) and a set $\Gamma_1$ of edges with an orientation $\mathfrak{o}$ on edges associated to every vertex. By viewing each edge as a set consisting of two half edges, we can redefine a vertex of $\Gamma$ as an (cyclic-)ordered set of half edges associated to a vertex where the (cyclic) order is induced from the original orientation $\mathfrak{o}$. In other words, we can define a Brauer graph in a new way as follows. Starting from a finite $G$-set $E$ (where $G=\langle g \rangle$ is an infinite cyclic group) whose elements are called half edges, we define the $G$-orbits as vertices of $E$ and the orientation on a vertex is given by the $G$-action, and define the edges of $E$ by a partition $P$ on $E$ such that each class $P(e)$ (here $P(e)$ denotes the equivalence class of $P$ containing $e\in E$) consisting of exactly two half edges. We can define the multiplicity function $\mu: E\rightarrow \mathbb{Z}_+$ directly on the set $E$ of half edges such that $\mu$ is constant on each $G$-orbit $G\cdot e$.

Using the above terminology, we recall that if $A_\Gamma$ is the associated Brauer graph algebra, then the indecomposable projective module associated to an edge $P(e)=\{e,e'\}$ is biserial and one uniserial sequence reads along the $G$-orbit $G\cdot e= \{e,g\cdot e, g^2\cdot e, \cdots\}$ by $\mu(e)$ times, and another uniserial sequence reads along the $G$-orbit $G\cdot e'$ by $\mu(e')$ times. We now use a new function $d: E\rightarrow \mathbb{Z}_+$ such that $d$ is constant on each $G$-orbit, which is called the degree function. And we can define a ``modified Brauer graph algebra" by defining the indecomposable projective module associated to an edge $P(e)=\{e,e'\}$ as follows: it is still biserial and one uniserial sequence reads along the $G$-orbit $G\cdot e=\{e,g\cdot e, g^2\cdot e, \cdots\}$ in $d(e)$ steps, and another uniserial sequence reads along the $G$-orbit $G\cdot e'$ in $d(e')$ steps. It is easy to see that a ``modified Brauer graph algebra" can be well-defined if we request a further condition: $P(e)=P(e')$ if and only if $P(g^{d(e)}\cdot e)=P(g^{d(e')}\cdot e')$ for all $e,e'\in E$. Note that a ``modified Brauer graph algebra" is a Frobenius algebra but not symmetric in general (see Example \ref{example0}). In Brauer graph algebra case, the multiplicity function $\mu(e)$ is equal to $\frac{d(e)}{\mid G\cdot e \mid}$ and $d(e)$ is an integral multiple of $\mid G\cdot e \mid$. However, in ``modified Brauer graph algebra" case, $\frac{d(e)}{\mid G\cdot e \mid}$ can be taken a fractional value, which will be called the fractional-degree at $e$ and is denoted by $d_f(e)$. The data $E=(E,P,d)$ defined as above will be called a fractional Brauer graph of type MS (see Section 3).

Similarly, a Brauer configuration $\Gamma=(\Gamma_0,\Gamma_1,\mu,\mathfrak{o})$ is defined by a set $\Gamma_0$ of vertices (with a multiplicity function $\mu: \Gamma_0\rightarrow \mathbb{Z}_+$ on it) and a set $\Gamma_1$ of polygons (each polygon is a multiset of vertices) with an orientation $\mathfrak{o}$ on polygons associated to every vertex. In order to define a fractional Brauer configuration, we start from a (finite or infinite) $G$-set $E$  (where $G=\langle g \rangle$ is an infinite cyclic group) whose elements are called angles. By similar ideas as above, we can define vertices of $E$ as $G$-orbits of the angles and polygons of $E$ by some partition $P$ of the angles (where $P(e)$ is a finite set for each $e\in E$) and associate a suitable degree function $d: E\rightarrow \mathbb{Z}_+$. We also use another partition $L$ of $E$ with the property that $L(e)\subseteq P(e)$ for each $e\in E$. The data $E=(E,P,L,d)$ defined as above will be called a fractional Brauer configuration. The precise definitions are given in Definition \ref{f-BC} and Remark \ref{remarks-on-the definition-of-f-BC}. Here we use angles rather than vertices and polygons as basic elements, since this avoids to use multisets and the angles can be directly used to define the arrows in the associated quiver of $E$ (see Section 4), and by this way it is also easier to define the morphisms between fractional Brauer configurations and the covering theory for fractional Brauer configurations (which are discussed in forthcoming papers \cite{LL2,LL3}).

We will define various types of fractional Brauer configurations. The types that we are most interested in are fractional Brauer configurations of type S (see Definition \ref{f-BC-type-S}), whose definition use some identification relation for sequences of angles in $E$ (see Definition \ref{parallel standard sequence}). The identification relation is important in our definition since it makes the associated category (or algebra) of $E$ not necessarily multiserial. The fractional Brauer configurations with trivial partition $L$ are called fractional Brauer configurations of type MS (see Definition \ref{f-BC-type-MS}) and they form an important subclass of fractional Brauer configurations of type S and include the fractional Brauer graphs of type MS introduced as above.

Our second step is to associate each fractional Brauer configuration $E=(E,P,L,d)$ a locally bounded $k$-category $\Lambda_E=kQ_E/I_E$ (where $Q_E$ is a locally finite quiver and $I_E$ is an ideal of the path category $kQ_E$), which will be called the fractional Brauer configuration category of $E$. Roughly speaking, the vertices set of $Q_{E}$ is given by polygons $\{P(e)\mid e\in E\}$, and the arrows set is given by $\{L(e)\mid e\in E\}$, where the arrow $L(e)$ has the source $P(e)$ and the terminal $P(g\cdot e)$. The ideal $I_E$ is defined by three kinds of relations which are similar to (but not equal to) defining relations for Brauer configuration algebras. The details are given in Definition \ref{f-BC-category} and in Definition \ref{ideal-I_E}. Note that in general $Q_E$ is not the Gabriel quiver of $\Lambda_E$ and $I_E$ is not an admissible ideal of $kQ_E$. In type S, it is not hard to represent $\Lambda_E$ as $kQ'_E/I'_E$ where $Q'_E$ is the Gabriel quiver of $\Lambda_E$ and $I'_E$ is an admissible ideal of $kQ'_E$ (see Lemma \ref{Gabriel quiver}). Similarly, we can define the fractional Brauer configuration algebra $A_E$ of $E$, which is a locally bounded $k$-algebra (see Definition \ref{f-BCA}).

Our third step is to prove various properties of fractional Brauer configuration categories and fractional Brauer configuration algebras. Among them, we state the following:
\begin{itemize}
	\item[(1)] If $E$ is a fractional Brauer configuration, then $\Lambda_E$ is a locally bounded $k$-category in the sense of Bongartz and Gabriel (see Theorem \ref{f-BC algebra is locally bounded}).
    \item[(2)] If $E$ is a fractional Brauer configuration of type S, then $\Lambda_E$ is a locally bounded Frobenius category (see Definition \ref{Frobenius category} and Theorem \ref{isomorphism of proj and inj}) so that the corresponding module category mod$\Lambda_E$ is a Frobenius category in the sense of Happel.
    \item[(3)] If $E$ is a fractional Brauer configuration of type MS, then $\Lambda_E$ is a locally bounded special multiserial Frobenius category (see Proposition \ref{special multiserial}).
    \item[(4)] If $E$ is a finite fractional Brauer configuration of type S with integral f-degree, then $A_E$ is a finite-dimensional symmetric algebra (see Proposition \ref{symmetric-algebra}). In particular, if $E$ is a finite fractional Brauer configuration of type MS with integral f-degree, then $A_E$ is a Brauer configuration algebra (see Corollary \ref{$f_{ms}$-BCA-of-integral-f-degree-is-equal-to-BCA}).

\medskip
\noindent Both (3)  and (4) can be seen as proper generalizations of the following result in \cite{GS2}: every Brauer configuration algebra is symmetric and special
multiserial.

\medskip
    \item[(5)] Over an algebraically closed field, the class of finite-dimensional representation-finite indecomposable fractional Brauer configuration algebras in type S coincides with the class of basic indecomposable finite-dimensional standard representation-finite self-injective algebras (see Theorem \ref{RFS-algebra=r.f.fsBCA}) and therefore this class of algebras is closed under derived equivalence.
\end{itemize}
	
We summarize the above discussion in the following table. Here we abbreviate Brauer graph algebra, Brauer configuration algebra, fractional Brauer configuration algebra, fractional Brauer configuration algebra in type S, fractional Brauer configuration algebra in type MS and fractional Brauer graph algebra in type MS by BGA, BCA, f-BCA, $f_s$-BCA, $f_{ms}$-BCA and $f_{ms}$-BGA, respectively.

\medskip
\begin{center}
    \begin{tabular}{|c|c|c|c|c|c|c|}
      \hline
       & BGA & BCA & $f_{ms}$-BGA & $f_{ms}$-BCA & $f_s$-BCA & f-BCA  \\
      \hline
      symmetric or not & yes & yes & no & no & no & no             \\
      \hline
      self-injective or not & yes & yes & yes & yes & yes & no             \\
      \hline
      special biserial or not & yes & no & yes & no & no & no             \\
      \hline
      special multiserial or not & yes & yes & yes & yes & no & no             \\
      \hline
    \end{tabular}
  \end{center}

\medskip
This paper is organized as follows. In section 2, we give a quick review on Brauer configuration and Brauer configuration algebra. In Section 3, 4 and 5, we define the fractional Brauer configurations, the associated fractional Brauer configuration categories, and the associated fractional Brauer configuration algebras respectively. In particular, in Section 4 we define the notion of locally bounded Frobenius category and show that the fractional Brauer configuration category in type S is a locally bounded Frobenius category. In Section 6, we determine the Gabriel quivers and admissible relations of fractional Brauer configuration categories in type S over an algebraically closed field. Based on this, we show that the fractional Brauer configuration algebra in type MS is a locally bounded special multiserial Frobenius algebra. In Section 7, based on the description of the standard form of a locally representation-finite category by Bretscher and Gabriel and the criterion to determining when a locally representation-finite category is standard by Mart\'inez-Villa and De La Pe\~na, we describe the class of finite-dimensional indecomposable representation-finite fractional Brauer configuration algebras in type S over an algebraically closed field.

\section*{Data availability} The datasets generated during the current study are available from the corresponding author on reasonable request.

\section*{Acknowledgements} This research is supported by NSFC (No.12031014).

\section{A quick review on Brauer configuration and Brauer configuration algebra}

Throughout this paper, let $k$ be a field. All algebras, modules and categories considered are $k$-linear. When we say a module $M$ over an algebra $A$ we always assume that $M$ is a left $A$-module unless otherwise stated. We recall (in a slightly new way) the definitions of a Brauer configuration and a Brauer configuration algebra from \cite[Section 1.1 and Section 2.2]{GS}.

\begin{Def}\label{BC}
A Brauer configuration (abbr. BC) is a quadruple $\Gamma=(\Gamma_0,\Gamma_1,\mu,\mathfrak{o})$ which is defined as follows.
\begin{itemize}
	\item[(1)] $\Gamma_0$ is a finite set, called the set of vertices;
    \item[(2)] $\Gamma_1$ is a finite collection of finite labelled multisets whose elements are in $\Gamma_0$, which is called the set of polygons, such that each vertex $\alpha\in\Gamma_0$ is contained in some polygon $V\in\Gamma_1$ and each polygon contains at least two vertices;
    \item[(3)] $\mu$ is a function $\Gamma_0\rightarrow\mathbb{Z}_{+}$, which is called the multiplicity function;
    \item[(4)] $\mathfrak{o}$ is an orientation of $\Gamma$, that is, for each vertex $\alpha\in\Gamma_0$, a cyclic order over the set $C_{\alpha}=\{(V,\alpha,i)\mid V\in\Gamma_1$ such that $\alpha$ occurs as a vertex in $V$ and $1\leq i\leq N_{V,\alpha}\}$, where $N_{V,\alpha}$ denotes the number of times that $\alpha$ occurs as a vertex in $V$.
    \item[(5)] For each polygon $V$, there exists a vertex $\alpha\in V$ such that val$(\alpha) \cdot \mu(\alpha)>1$, where val$(\alpha)$ denotes the cardinal of the set $C_{\alpha}$.
\end{itemize}
\end{Def}

\begin{Rem1} \label{remarks-on-BC}
\begin{itemize}
	\item[(1)] We say that a vertex $\alpha\in\Gamma_0$ is truncated if val$(\alpha)\mu(\alpha)=1$; that is, $\alpha$ occurs exactly once in exactly one $V\in\Gamma_1$ and $\mu(\alpha)=1$.
    \item[(2)] A Brauer configuration is called a Brauer graph (abbr. BG) if each polygon of it contains exactly two vertices.
    \item[(3)] Note that in the above definition we have rewritten the orientation in Green and Schroll's original definition using the set $C_\alpha$. We will call the elements of $C_\alpha$ as angles at the vertex $\alpha$ (see and compare Definition \ref{BC-reformulation} below) and this viewpoint is critical in our definition of fractional Brauer configuration.
    \item[(4)] The basic elements in Definition \ref{BC} are vertices and polygons and the new notion of angles are derived from them.
\end{itemize}
\end{Rem1}

\begin{Ex1} \label{example-BC} Let $\Gamma=(\Gamma_0,\Gamma_1,\mu,\mathfrak{o})$ be a BC defined as follows. Let
$$\Gamma_0=\{1, 2, 3\},$$
$$\Gamma_1=\{V_1=\{1, 1, 3\}, V_2=\{1, 2\}\}.$$
We choose $\mu(3)=2$, and $\mu(i)=1$ for all other vertices. The orientation is given by: $(V_1,1,1)<(V_1,1,2)<(V_2,1,1)<(V_1,1,1)$ on $C_1$, $(V_2,2,1) < (V_2,2,1)$ on $C_2$, $(V_1,3,1) < (V_1,3,1)$ on $C_3$.

The above BC can be realized by the following diagram:
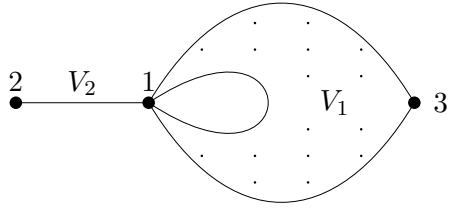
\begin{center}
\tikzset{every picture/.style={line width=0.25pt}} %set default line width to 0.75pt
\begin{tikzpicture}[x=10pt,y=10pt,yscale=1,xscale=1]
%uncomment if require: \path (0,314); %set diagram left start at 0, and has height of 314

\draw    (0,0) -- (-5,0) ;
\draw    (0,0) .. controls (3,5) and (7,5) .. (10,0) ;
\draw    (0,0) .. controls (3,-5) and (7,-5) .. (10,0) ;
\draw    (0,0) .. controls (6,4) and (6,-4) .. (0,0) ;
\draw (0,0.8) node [inner sep=0.75pt]    {$1$};
\draw (-5,0.8) node [inner sep=0.75pt]    {$2$};
\draw (11,0) node [inner sep=0.75pt]    {$3$};
\draw (7,0) node [inner sep=0.75pt]    {$V_1$};
\draw (-2.5,0.75) node [inner sep=0.75pt]    {$V_2$};
\fill (0,0) circle (0.5ex);
\fill (-5,0) circle (0.5ex);
\fill (10,0) circle (0.5ex);
\fill (4,3) circle (0.1ex);
\fill (6,3) circle (0.1ex);
\fill (2,2) circle (0.1ex);
\fill (4,2) circle (0.1ex);
\fill (6,2) circle (0.1ex);
\fill (8,2) circle (0.1ex);
\fill (6,1) circle (0.1ex);
\fill (8,1) circle (0.1ex);
\fill (4,-3) circle (0.1ex);
\fill (6,-3) circle (0.1ex);
\fill (2,-2) circle (0.1ex);
\fill (4,-2) circle (0.1ex);
\fill (6,-2) circle (0.1ex);
\fill (8,-2) circle (0.1ex);
\fill (6,-1) circle (0.1ex);
\fill (8,-1) circle (0.1ex);

\end{tikzpicture}
\end{center}
\end{Ex1}

Green and Schroll associate each BC $\Gamma$ a quiver algebra $kQ_{\Gamma}'/I_{\Gamma}'$ in \cite{GS}. Since we do not emphasize the truncated vertices in our formulation, the quiver $Q_{\Gamma}$ and the ideal $I_{\Gamma}$ defined below is slightly different from that in  \cite{GS}, in particular, here $I_{\Gamma}$ is not an admissible ideal in general. However, the quotient algebras $kQ_{\Gamma}/I_{\Gamma}$ and $kQ_{\Gamma}'/I_{\Gamma}'$ are isomorphic.

The vertices of $Q_{\Gamma}$ are in one to one correspondence with the polygons of $\Gamma$. For each vertex $\alpha\in\Gamma_0$ and for $(V_1,\alpha,i_1)$, $(V_2,\alpha,i_2)\in C_{\alpha}$ such that $(V_2,\alpha,i_2)$ is the successor of $(V_1,\alpha,i_1)$, there is an arrow from $v_1$ to $v_2$ in $Q_{\Gamma}$, where $v_i$ is the vertex of $Q_{\Gamma}$ which correspond to the polygon $V_i$ of $\Gamma$ for $i=1,2$.

Let $a$ be an arrow of $Q_{\Gamma}$, which corresponds to the triple $(V',\alpha,i')$ being the successor of the triple $(V,\alpha,i)$ in $C_{\alpha}$. Let the triple $(V'',\alpha,i'')\in C_{\alpha}$ be the successor of the triple $(V',\alpha,i')$ and let $a'$ be an arrow of $Q_{\Gamma}$ which corresponds to $(V'',\alpha,i'')$ being the successor of $(V',\alpha,i')$ in $C_{\alpha}$. Define a bijection $\iota:(Q_{\Gamma})_1\rightarrow (Q_{\Gamma})_1$ by mapping $a$ to $a'$.

For each arrow $a$ of $Q_{\Gamma}$ which corresponds to the triple $(V',\alpha,i')$ being the successor of the triple $(V,\alpha,i)$ in $C_{\alpha}$, define $p(a)=\iota^{n-1}(a)\cdots\iota(a)a$ to be a cycle in the quiver $Q_{\Gamma}$, where $n$ is the smallest positive integer such that $\iota^{n}(a)=a$ ($n$ equals to the cardinal of $C_{\alpha}$). Call $p(a)$ a special $\alpha$-cycle. Let $I_{\Gamma}$ be the ideal of the path algebra $kQ_{\Gamma}$ generated by the following two types of relations:
\begin{itemize}
	\item[$(R1)$] $p(a)^{\mu(\alpha)}-p(b)^{\mu(\beta)}$, where $a$, $b$ are arrows of $Q_{\Gamma}$ which start at the same vertex, and $p(a)$ (resp. $p(b)$) is a special $\alpha$-cycle (resp. a special $\beta$-cycle);
    \item[$(R2)$] $ba$, where $a$, $b$ are arrows of $Q_{\Gamma}$ such that $b\neq\iota(a)$.
\end{itemize}

\begin{Def}{\rm (cf. \cite[Definition 2.5]{GS})} \label{BCA}
The Brauer configuration algebra (abbr. BCA) $\Lambda_{\Gamma}$ associated to $\Gamma$ is defined to be the quotient algebra $kQ_{\Gamma}/I_{\Gamma}$.
\end{Def}

Note that each Brauer configuration algebra is a special multiserial algebra and in particular a multiserial algebra (see \cite{GS2}).

\begin{Ex1}\label{example-BCA} The BCA corresponding to the BC (which we denote by $\Gamma$) in Example \ref{example-BC} is given by the following quiver $Q_{\Gamma}$
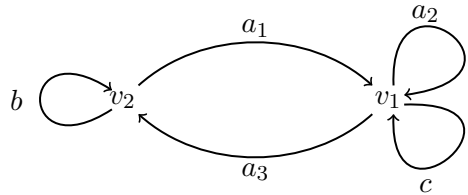
\begin{center}
\tikzset{every picture/.style={line width=0.75pt}}
\begin{tikzpicture}[x=20pt,y=18pt,yscale=1,xscale=1]
\node at(0,0) {$v_2$};
\node at(5,0) {$v_1$};
\node at(2.5,1.5) {$a_1$};
\node at(5.7,1.8) {$a_2$};
\node at(2.5,-1.5) {$a_3$};
\node at(-2,0) {$b$};
\node at(5.7,-1.8) {$c$};
\draw[->]    (0.3,0.3) .. controls (1.5,1.5) and (3.5,1.5) .. (4.7,0.3) ;
\draw[->]    (4.7,-0.3) .. controls (3.5,-1.5) and (1.5,-1.5) .. (0.3,-0.3) ;
\draw[->]    (5.1,0.3) .. controls (5,3) and (8,0.5) .. (5.3,0.1) ;
\draw[->]    (5.3,-0.1) .. controls (8,0) and (5,-3) .. (5.1,-0.3) ;
\draw[->]    (-0.2,-0.2) .. controls (-2,-1.5) and (-2,1.5) .. (-0.2,0.2) ;
\end{tikzpicture}
\end{center}
with relations
\begin{multline*}a_3a_2a_1=b, a_1a_3a_2=a_2a_1a_3=c^2, 0=c a_1=a_3a_1=c a_2=a_{2}^{2}=b a_3=a_1 b=a_2 c=a_3 c.\end{multline*}
Therefore, the indecomposable projective modules have the following structures:

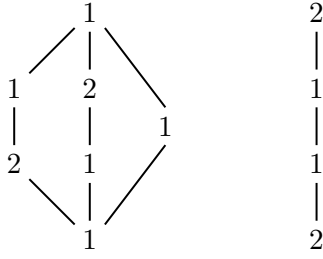
\begin{center}
\tikzset{every picture/.style={line width=0.75pt}}
\begin{tikzpicture}
\node at(0,0) {$1$};
\node at(-1,1) {$2$};
\node at(-1,2) {$1$};
\node at(0,3) {$1$};
\node at(0,1) {$1$};
\node at(0,2) {$2$};
\node at(1,1.5) {$1$};
\draw    (-0.2,0.2) -- (-0.8,0.8) ;
\draw    (-1,1.25) -- (-1,1.75) ;
\draw    (-0.8,2.2) -- (-0.2,2.8) ;
\draw    (0,0.25) -- (0,0.75) ;
\draw    (0,1.25) -- (0,1.75) ;
\draw    (0,2.25) -- (0,2.75) ;
\draw    (0.2,0.2) -- (1,1.25) ;
\draw    (1,1.75) -- (0.2,2.8) ;
\node at(3,0) {$2$};
\node at(3,1) {$1$};
\node at(3,2) {$1$};
\node at(3,3) {$2$};
\draw    (3,0.25) -- (3,0.75) ;
\draw    (3,1.25) -- (3,1.75) ;
\draw    (3,2.25) -- (3,2.75) ;
\end{tikzpicture}
\end{center}
\end{Ex1}

\section{Fractional Brauer configurations}

In this section we define fractional Brauer configurations, which are generalization of Brauer configurations. According to \cite{S2}, a Brauer graph can be seen as a ribbon graph with a multiplicity function. Inspired by this, let us first reformulate the definition of a Brauer configuration as follows.

\begin{Def} \label{BC-reformulation}
A Brauer configuration (abbr. BC) is a tuple $\Delta=(\Delta_0,\Delta_1,\zeta,P,\mathfrak{p},\nu)$ which is defined as follows.
\begin{itemize}
	\item[(1)] $\Delta_0$ is a finite set, called the set of vertices;
    \item[(2)] $\Delta_1$ is a finite set, called the set of angles;
    \item[(3)] $\zeta:\Delta_1\rightarrow\Delta_0$ is a surjective map, which is called the connection map;
    \item[(4)] $P$ is a partition of $\Delta_1$, such that $|P(e)|\geq 2$ for every $e\in\Delta_1$ ($P(e)$ is the equivalence class determined by $P$ which contains the angle $e$, which is called a polygon);
    \item[(5)] $\mathfrak{p}$ is a permutation $\Delta_1\rightarrow\Delta_1$ whose cycles have underlying sets $\zeta^{-1}(\alpha)$ with $\alpha\in\Delta_0$, which is called an orientation of $\Delta$;
    \item[(6)] $\nu:\Delta_0\rightarrow\mathbb{Z}_{+}$ is a function, which is called the multiplicity function;
    \item[(7)] For every polygon $P(e)$ of $E$, there exists some $h\in P(e)$ with $\mid\zeta^{-1}(\zeta(h))\mid\cdot\nu(\zeta(h))>1$.
\end{itemize}
	
A Brauer configuration is called a Brauer graph if each polygon of it contains exactly two angles.
\end{Def}

Compare with Definition \ref{BC}, in the above definition we use vertices and angles as the basic elements, and the polygons are derived from them.

\begin{Prop}\label{two-definitions-are-equivalent}
Definition \ref{BC-reformulation} and Definition \ref{BC} are equivalent.
\end{Prop}

\begin{proof}
Let $\Gamma=(\Gamma_0,\Gamma_1,\mu,\mathfrak{o})$ be a Brauer configuration in Definition \ref{BC}, construct a tuple $\Delta=(\Delta_0,\Delta_1,\zeta,P,\mathfrak{p},\nu)$ as follows: $\Delta_0=\Gamma_0$; $\Delta_1=\bigcup_{\alpha\in\Gamma_0}C_{\alpha}$, where $C_{\alpha}=\{(V,\alpha,i)\mid V\in\Gamma_1$ such that $\alpha$ occurs as a vertex in $V$ and $1\leq i\leq N_{V,\alpha}\}$, $N_{V,\alpha}$ denotes the number of times that $\alpha$ occurs as a vertex in $V$; $\zeta:\Delta_1\rightarrow\Delta_0$ is given by $(V,\alpha,i)\mapsto\alpha$; $P$ is a partition of $\Delta_1$ such that for each angle $(V,\alpha,i)$ in $\Delta_1$, the class $P(V,\alpha,i)$ is equal to $\{(V,\beta,j)\mid\beta\in\Gamma_0$ occurs as a vertex in $V$ and $1\leq j\leq N_{V,\beta}\}$, which is determined by the polygon $V$; $\mathfrak{p}$ is a permutation $\Delta_1\rightarrow\Delta_1$ such that $\mathfrak{p}(V,\alpha,i)=(V',\alpha,i')$, where $(V',\alpha,i')$ is the successor of the triple $(V,\alpha,i)$ under the cyclic order $\mathfrak{o}$ on $C_\alpha$; $\nu=\mu:\Delta_0\rightarrow\mathbb{Z}_{+}$. It is straightforward to show that $\Delta=(\Delta_0,\Delta_1,\zeta,P,\mathfrak{p},\nu)$ is a Brauer configuration under Definition \ref{BC-reformulation}. Note that the condition $(5)$ in Definition \ref{BC} corresponds to the condition $(7)$ in Definition \ref{BC-reformulation}.

Conversely, let $\Delta=(\Delta_0,\Delta_1,\zeta,P,\mathfrak{p},\nu)$ be Brauer configuration in Definition \ref{BC-reformulation}, we define a quadruple $\Gamma=(\Gamma_0,\Gamma_1,\mu,\mathfrak{o})$ as follows: $\Gamma_0=\Delta_0$; $\Gamma_1=\{[\zeta(P(e))]\mid e\in\Delta_1\}$, where $[\zeta(P(e))]$ represents the finite labelled multiset which is the image of the class $P(e)$ under $\zeta$; $\mu=\nu:\Gamma_0\rightarrow\mathbb{Z}_{+}$; In order to define the orientation $\mathfrak{o}$, we first define the set $C_{\alpha}$ of angles at the vertex $\alpha$ (in the sense of Definition \ref{BC}) as follows.

For each vertex $\alpha\in\Gamma_0$, define the set $C_{\alpha}$ to be $\{(V,\alpha,i)\mid V\in\Gamma_1$ such that $\alpha$ appears in $V=[\zeta(P(e))]$ for some $e\in \Delta_1$ and $1\leq i\leq N_{V,\alpha}\}$, where $N_{V,\alpha}$ denotes the number of times that $\alpha$ appears in $V$. Note that for a polygon $V=[\zeta(P(e))]$, $\alpha\in V$ if and only if $\zeta^{-1}(\alpha)\cap P(e)\neq \emptyset$, and $N_{V,\alpha}$ is equal to the cardinal of $\zeta^{-1}(\alpha)\cap P(e)$. For each polygon $P(e)$ of $\Delta$ with $\zeta^{-1}(\alpha)\cap P(e)\neq \emptyset$, label the elements of $\zeta^{-1}(\alpha)\cap P(e)$ by $e_1,\cdots,e_s$, and define a map $\phi_{P(e)}:\zeta^{-1}(\alpha)\cap P(e)\rightarrow C_{\alpha}$ by $\phi_{P(e)}(e_i)=([\zeta(P(e))],\alpha,i)$. Since $\zeta^{-1}(\alpha)$ is a disjoint union of subsets of the form $\zeta^{-1}(\alpha)\cap P(e)$, the maps $\phi_{P(e)}:\zeta^{-1}(\alpha)\cap P(e)\rightarrow C_{\alpha}$ induce a map $\phi:\zeta^{-1}(\alpha)\rightarrow C_{\alpha}$, which is bijective. Then the cyclic order $\mathfrak{o}$ on $C_{\alpha}$ is induced by the permutation $\mathfrak{p}$ on $\zeta^{-1}(\alpha)$. Also the condition $(7)$ in Definition \ref{BC-reformulation} corresponds to the condition $(5)$ in Definition \ref{BC}.
\end{proof}

In view of Definition \ref{BC-reformulation}, we can also reformulate the definition of a Brauer configuration algebra as follows. For a Brauer configuration $\Delta=(\Delta_0,\Delta_1,\zeta,P,\mathfrak{p},\nu)$, define a quiver $Q_{\Delta}$ as follows:
\begin{itemize}
	\item The vertices of $Q_{\Delta}$ correspond to the polygons of $\Delta$;
    \item The arrows of $Q_{\Delta}$ are given by the angles of $\Delta$ by the following way: given an angle $e\in \Delta_1$ with connection vertex $\alpha=\zeta(e)\in \Delta_0$, there is an arrow from the polygon $P(e)$ to the polygon $P(h)$, where $h=\mathfrak{p}(e)$, that is, $h$ is the successor angle of $e$ under the permutation around the vertex $\alpha$.
\end{itemize}
In the following, we will identify the angle $e$ with the corresponding arrow in $Q_{\Delta}$ and write $s(e)=P(e), t(e)=P(h)$ for the arrow $e$. Let $e$ be an arrow of $Q_{\Delta}$, define $p(e)=\mathfrak{p}^{n-1}(e)\cdots\mathfrak{p}(e)e$ to be the cycle in the quiver $Q_{\Delta}$, where $n$ is the smallest positive integer such that $\mathfrak{p}^{n}(e)=e$. Let $I_{\Delta}$ be the ideal of the path algebra $kQ_{\Delta}$ generated by the following two types of relations:
\begin{itemize}
	\item[$(R1)$] $p(e)^{\nu(\zeta(e))}-p(h)^{\nu(\zeta(h))}$, where $e$, $h$ are arrows of $Q_{\Delta}$ with $P(e)=P(h)$;
    \item[$(R2)$] $e_2 e_1$, where $e_1$, $e_2$ are arrows of $Q_{\Delta}$ such that $e_2\neq\mathfrak{p}(e_1)$.
\end{itemize}
The Brauer configuration algebra $\Lambda_{\Delta}$ associated to $\Delta$ is defined to be $kQ_{\Delta}/I_{\Delta}$.

\medskip
We now introduce the notion of fractional Brauer configuration.

\begin{Def}\label{f-BC}
Let $G=\langle g \rangle$ be an infinite cyclic group. A fractional Brauer configuration (abbr. f-BC) is a quadruple $E=(E,P,L,d)$, where $E$ is a $G$-set, $P$ and $L$ are two partitions of $E$, and $d: E\rightarrow \mathbb{Z}_{+}$ is a function, such that the following conditions hold.
\begin{itemize}
	\item[$(f1)$] $L(e)\subseteq P(e)$ and $P(e)$ is a finite set for each $e\in E$.
    \item[$(f2)$] If $L(e_1)=L(e_2)$, then $P(g\cdot e_1)=P(g\cdot e_2)$.
    \item[$(f3)$] If $e_1$, $e_2$ belong to same $\langle g \rangle$-orbit, then $d(e_1)=d(e_2)$.
    \item[$(f4)$] $P(e_1)=P(e_2)$ if and only if $P(g^{d(e_1)}\cdot e_1)=P(g^{d(e_2)}\cdot e_2)$.
    \item[$(f5)$] $L(e_1)=L(e_2)$ if and only if $L(g^{d(e_1)}\cdot e_1)=L(g^{d(e_2)}\cdot e_2)$.
    \item[$(f6)$] The formal sequence $L(g^{d(e)-1}\cdot e)\cdots L(g\cdot e)L(e)$ is not a proper subsequence of the formal sequence $L(g^{d(h)-1}\cdot h)\cdots L(g\cdot h)L(h)$ for all $e,h\in E$.
\end{itemize}
\end{Def}

\begin{Rem1} \label{remarks-on-the definition-of-f-BC} We need some further remarks on the above definition as follows.
\begin{itemize}
	\item[(1)] The elements in $E$ are called angles of the f-BC. The $\langle g \rangle$-orbits of $E$ are called vertices of the f-BC. In other words, every vertex of the f-BC is of the form $G\cdot e$ where $e\in E$ is an angle.
    \item[(2)] For each vertex $v$, the $\langle g \rangle$-set structure of $v$ gives an order on it, which is a total order when $v$ is infinite and is a cyclic order when $v$ is finite. The order given by $\langle g\rangle$-action on each vertex of f-BC can be regarded as a generalization of the orientation $\mathfrak{o}$ of BC.
    \item[(3)] The classes $P(e)$ of the partition $P$ are called polygons. Note that we allow the cardinality of $P(e)$ to be 1, that is, $P(e)$ can be a 1-gon, which is different from the BC case.
    \item[(4)] Condition $(f1)$ means that any two angles $e_1,e_2$ of the class $L(e)$ lie in the same polygon $P(e)$, and Condition $(f2)$ means that their successors $g\cdot e_1, g\cdot e_2$ also lie in the same polygon $P(g\cdot e)$. The partition $L$ is said to be trivial if $L(e)=\{e\}$ for each $e\in E$.
    \item[(5)] The function $d: E\rightarrow\mathbb{Z}_{+}$ is called degree function. Condition $(f3)$ means that the degree function can be defined on vertices.
    \item[(6)] Let $E$ be an f-BC and $v$ be a vertex such that $v$ is a finite set, define the fractional-degree (abbr. f-degree) $d_f(v)$ of the vertex $v$ to be the rational number $\frac{d(v)}{\mid v\mid}$. If the f-degree of each vertex of $E$ is an integer, then $E$ is said to have integral f-degree. Moreover, if $d_f(v)=1$ for each vertex $v$, then $E$ is called f-degree trivial.
    \item[(7)] Denote $\sigma$ the map $E\rightarrow E$, $e\mapsto g^{d(e)}\cdot e$, which will be called the Nakayama automorphism of $E$. This automorphism is guaranteed by the conditions $(f3)$, $(f4)$ and $(f5)$.
\end{itemize}
\end{Rem1}

Each BC $\Delta=(\Delta_0,\Delta_1,\zeta,P,\mathfrak{p},\nu)$ can be seen as an f-BC $E=(E,P,L,d)$ with integral f-degree and trivial partition $L$ by the following procedure: Let $E=\Delta_1$ such that the $\langle g \rangle$-action on $E$ is induced by the permutation $\mathfrak{p}$ on $\Delta_1$; the partition $P$ of $E$ is just the partition $P$ of $\Delta_1$; the partition $L$ of $E$ is defined to be trivial; the degree function $d$ is given by $d(e)=\nu(\zeta(e))\cdot\mid\zeta^{-1}(\zeta(e))\mid$. Then the multiplicity function $\nu$ is equal to the f-degree $d_f$.

On the other hand, if we allow the multiplicity function $\nu$ in a BC $\Delta=(\Delta_0,\Delta_1,\zeta,P,\mathfrak{p},\nu)$ to take fractional values such that $\nu(v)$ is an integral multiple of $\frac{1}{\mid\zeta^{-1}(v)\mid}$ for each vertex $v\in\Delta_0$, then under some suitable conditions we can still form an f-BC along the above procedure (see Example \ref{example0}).  Note that in BC case, a vertex $v$ is truncated (cf. Remark \ref{remarks-on-BC} (1)) if and only if $d(v)=1$.

\medskip
We illustrate our definition of f-BC with some examples.

\begin{Ex1}\label{example0}
Let $E=\{1,1',2,2',3,3'\}$. Define the group action on $E$ by $g\cdot 1=2$, $g\cdot 2=3$, $g\cdot 3=1$, $g\cdot 1'=2'$, $g\cdot 2'=3'$, $g\cdot 3'=1'$. Define $P(1)=\{1,1'\}$, $P(2)=\{2,2'\}$, $P(3)=\{3,3'\}$ and $L(e)=\{e\}$ for every $e\in E$. The degree $d$ of $E$ is defined by $d(e)=2$ for every $e\in E$. So the f-degree of $E$ has constant value $\frac{2}{3}$.
\end{Ex1}

\begin{Ex1}\label{example1}
Let $E=\{1,1',2,2',3,3',4,4'\}$. Define the group action on $E$ by $g\cdot 1=2$, $g\cdot 2=3$, $g\cdot 3=1$, $g\cdot 1'=2'$, $g\cdot 2'=4'$, $g\cdot 4'=1'$, $g\cdot 3'=3'$, $g\cdot 4=4$. Define $P(1)=\{1,1'\}$, $P(2)=\{2,2'\}$, $P(3)=\{3,3'\}$, $P(4)=\{4,4'\}$, $L(1)=\{1,1'\}$ and $L(e)=\{e\}$ for $e\neq 1,1'$. The f-degree of $E$ is defined to be trivial.
\end{Ex1}

\begin{Ex1}\label{example2}
Let $E=\{a_1,a_2,a_3,a_4\}$. Define the group action on $E$ by $g\cdot a_1=a_1$, $g\cdot a_2=a_3$, $g\cdot a_3=a_2$, $g\cdot a_4=a_4$. Define $P(a_1)=E$, $L(a_1)=\{a_1,a_2\}$, $L(a_3)=\{a_3,a_4\}$. Define $d(a_1)=d(a_4)=2$, $d(a_2)=d(a_3)=4$. So the f-degree of $E$ has constant value $2$.
\end{Ex1}

\begin{Ex1}\label{example3}
Let $E=\{1,1',1'',2,2',3,4,4',4'',5,5',6\}$. Define the group action on $E$ by $g\cdot 1=2$, $g\cdot 2=4$, $g\cdot 4=5$, $g\cdot 5=1$, $g\cdot 1'=2'$, $g\cdot 2'=4'$, $g\cdot 4'=6$, $g\cdot 6=1'$, $g\cdot 1''=3$, $g\cdot 3=4''$, $g\cdot 4''=5'$, $g\cdot 5'=1''$. Define $P(1)=\{1,1',1''\}$, $P(2)=\{2,2'\}$, $P(3)=\{3\}$, $P(4)=\{4,4',4''\}$, $P(5)=\{5,5'\}$, $P(6)=\{6\}$, $L(1)=\{1,1'\}$, $L(2)=\{2,2'\}$, $L(4)=\{4,4''\}$, $L(5)=\{5,5'\}$, $L(e)=\{e\}$ for other $e\in E$. The f-degree of $E$ is defined to be trivial.
\end{Ex1}

\begin{Ex1}\label{example4}
Let $E=\{i\mid i\in\mathbb{Z}\}\cup\{i'\mid i\in 3\mathbb{Z}\}$. Define the group action on $E$ by
$$
g\cdot e= \begin{cases}
i+1, &  \text{ if } e=i \text{ with } i\equiv 0 \text{ }(\text{mod } 3); \\
i+2, &  \text{ if } e=i \text{ with } i\equiv 1 \text{ }(\text{mod } 3) \text{ or if } e=i' \text{ with } i\equiv 0 \text{ }(\text{mod } 3); \\
(i+1)', &  \text{ if } e=i \text{ with } i\equiv 2 \text{ }(\text{mod } 3).
\end{cases}
$$
Define
$$
P(e)= \begin{cases}
\{i,i'\}, &  \text{ if } e=i \text{ with } i\equiv 0 \text{ }(\text{mod } 3); \\
\{e\}, &  \text{ otherwise};
\end{cases}
$$
and $L(e)=\{e\}$ for every $e\in E$. The degree $d$ of $E$ is given by $d(e)=2$ for every $e\in E$.
\end{Ex1}

Every fractional Brauer configuration $E=(E,P,L,d)$ can be visualized as some diagram $\Gamma(E)$, which consists of polygons such that every two polygons can only intersect on their vertices. Each polygon in $\Gamma(E)$ corresponds to a polygon (of the form $P(e)$) of $E$, and each $e\in E$ corresponds to an angle of the polygon $P(e)$ in $\Gamma(E)$. Two angles of polygons in $\Gamma(E)$ are connected to the same vertex of $\Gamma(E)$ if and only if the two corresponding elements of $E$ belong to the same $\langle g \rangle$-orbit of $E$. For each vertex $v=G\cdot e$ of $\Gamma(E)$, the $\langle g\rangle$-set structure of $E$ gives an order (which is always taken to be clockwise) to angles of polygons in $\Gamma(E)$ which are connected to $v$.

Note that if a polygon $P(e)$ has cardinality 1, then $P(e)$ has only one angle and it corresponds to a half edge in $\Gamma(E)$; and if $P(e)$ has cardinality 2, then $P(e)$ has two angles and they correspond to one edge (or two half edges) in $\Gamma(E)$. The f-BCs in Examples \ref{example0}, \ref{example1}, \ref{example2}, \ref{example3}, \ref{example4} can be visualized by the following diagrams

\vspace{0.5cm}
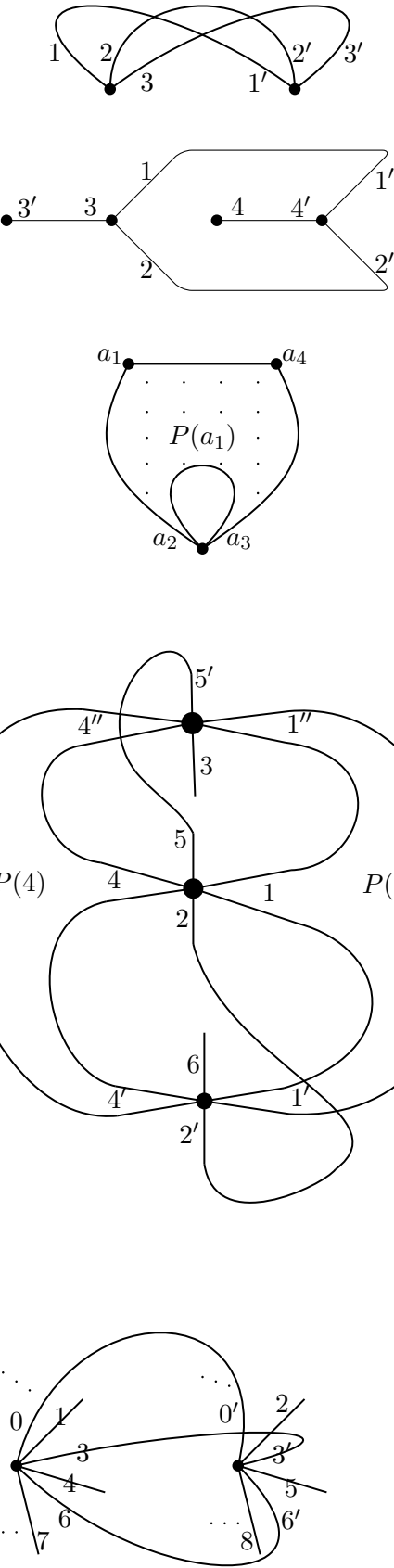
\begin{center}
\tikzset{every picture/.style={line width=0.75pt}}
\begin{tikzpicture}[x=15pt,y=15pt,yscale=1,xscale=1]
\fill (0,0) circle (0.5ex);
\fill (5,0) circle (0.5ex);
\node at(-1.5,1) {$1$};
\node at(-0.1,1) {$2$};
\node at(1,0.2) {$3$};
\node at(6.6,1) {$3'$};
\node at(5.2,1) {$2'$};
\node at(4,0.2) {$1'$};
\draw    (0,0) .. controls (-4,3) and (1,3) .. (5,0) ;
\draw    (0,0) .. controls (0,3) and (5,3) .. (5,0) ;
\draw    (0,0) .. controls (4,3) and (9,3) .. (5,0) ;
\end{tikzpicture}
\end{center}

\vspace{0.5cm}
\begin{center}
\begin{tikzpicture}
\draw[rounded corners] (1,-1)--(2,0)--(5,0)--(4,-1);
\draw[rounded corners] (1,-1)--(2,-2)--(5,-2)--(4,-1);
\draw (-0.5,-1)--(1,-1);
\draw (2.5,-1)--(4,-1);
\node at(1.5,-0.3) {$1$};
\node at(4.9,-0.4) {$1'$};
\node at(1.5,-1.7) {$2$};
\node at(4.9,-1.6) {$2'$};
\node at(-0.2,-0.8) {$3'$};
\node at(0.7,-0.8) {$3$};
\node at(2.8,-0.8) {$4$};
\node at(3.7,-0.8) {$4'$};
\fill (1,-1) circle (0.5ex);
\fill (4,-1) circle (0.5ex);
\fill (-0.5,-1) circle (0.5ex);
\fill (2.5,-1) circle (0.5ex);
\end{tikzpicture}
\end{center}

\vspace{0.5cm}
\begin{center}
\tikzset{every picture/.style={line width=0.75pt}}
\begin{tikzpicture}[x=15pt,y=15pt,yscale=1,xscale=1]
\fill (-2,5) circle (0.5ex);
\fill (2,5) circle (0.5ex);
\fill (0,0) circle (0.5ex);
\draw    (-2,5) -- (2,5) ;
\draw    (0,0) .. controls (-3,2) and (-3,3) .. (-2,5) ;
\draw    (0,0) .. controls (3,2) and (3,3) .. (2,5) ;
\draw    (0,0) .. controls (-3,3) and (3,3) .. (0,0) ;
\node at(-2.5,5.2) {$a_1$};
\node at(2.5,5.2) {$a_4$};
\node at(-1,0.2) {$a_2$};
\node at(1,0.2) {$a_3$};
\node at(0,3) {$P(a_1)$};
\fill (-1.5,4.5) circle (0.1ex);
\fill (-0.5,4.5) circle (0.1ex);
\fill (0.5,4.5) circle (0.1ex);
\fill (1.5,4.5) circle (0.1ex);
\fill (-1.5,3.7) circle (0.1ex);
\fill (-0.5,3.7) circle (0.1ex);
\fill (0.5,3.7) circle (0.1ex);
\fill (1.5,3.7) circle (0.1ex);
\fill (-1.5,3) circle (0.1ex);
\fill (1.5,3) circle (0.1ex);
\fill (-1.5,2.3) circle (0.1ex);
\fill (-0.5,2.3) circle (0.1ex);
\fill (0.5,2.3) circle (0.1ex);
\fill (1.5,2.3) circle (0.1ex);
\fill (-1.5,1.5) circle (0.1ex);
\fill (1.5,1.5) circle (0.1ex);
\end{tikzpicture}
\end{center}

\vspace{0.5cm}
\begin{center}
\tikzset{every picture/.style={line width=0.75pt}} %set default line width to 0.75pt

\begin{tikzpicture}[x=0.75pt,y=0.75pt,yscale=-1,xscale=1]
%uncomment if require: \path (0,378); %set diagram left start at 0, and has height of 378

%Shape: Circle [id:dp6482326854816789]
\draw  [fill={rgb, 255:red, 0; green, 0; blue, 0 }  ,fill opacity=1 ] (320,97.5) .. controls (320,94.46) and (322.46,92) .. (325.5,92) .. controls (328.54,92) and (331,94.46) .. (331,97.5) .. controls (331,100.54) and (328.54,103) .. (325.5,103) .. controls (322.46,103) and (320,100.54) .. (320,97.5) -- cycle ;
%Shape: Circle [id:dp8722661582167257]
\draw  [fill={rgb, 255:red, 0; green, 0; blue, 0 }  ,fill opacity=1 ] (321,187) .. controls (321,184.24) and (323.24,182) .. (326,182) .. controls (328.76,182) and (331,184.24) .. (331,187) .. controls (331,189.76) and (328.76,192) .. (326,192) .. controls (323.24,192) and (321,189.76) .. (321,187) -- cycle ;
%Shape: Circle [id:dp5415159913450924]
\draw  [fill={rgb, 255:red, 0; green, 0; blue, 0 }  ,fill opacity=1 ] (328,302) .. controls (328,299.79) and (329.79,298) .. (332,298) .. controls (334.21,298) and (336,299.79) .. (336,302) .. controls (336,304.21) and (334.21,306) .. (332,306) .. controls (329.79,306) and (328,304.21) .. (328,302) -- cycle ;
%Straight Lines [id:da2554891363429517]
\draw    (326,187) -- (383,206) ;
%Straight Lines [id:da22139437420669328]
\draw    (326,187) -- (379,177) ;
%Straight Lines [id:da4735446660042979]
\draw    (276,173) -- (326,187) ;
%Straight Lines [id:da13338082589390976]
\draw    (279,194) -- (326,187) ;
%Straight Lines [id:da2747683361657176]
\draw    (285,310) -- (332,302) ;
%Straight Lines [id:da3403126433936523]
\draw    (284,294) -- (332,302) ;
%Straight Lines [id:da02809507010813328]
\draw    (375,295) -- (332,302) ;
%Straight Lines [id:da4325077468969849]
\draw    (332,302) -- (378,309) ;
%Straight Lines [id:da9035206605704178]
\draw    (263,110) -- (325.5,97.5) ;
%Straight Lines [id:da6679758526191082]
\draw    (266,90) -- (325.5,97.5) ;
%Straight Lines [id:da8856897242797797]
\draw    (325.5,97.5) -- (376,108) ;
%Straight Lines [id:da900913512255457]
\draw    (325.5,97.5) -- (378,91) ;
%Curve Lines [id:da6254252667647553]
\draw    (376,108) .. controls (437,116) and (418,176) .. (379,177) ;
%Curve Lines [id:da2106549937194997]
\draw    (375,295) .. controls (443,275) and (432,219) .. (383,206) ;
%Curve Lines [id:da8633859163784772]
\draw    (378,91) .. controls (495,81) and (516,317) .. (378,309) ;
%Curve Lines [id:da39035702407704753]
\draw    (276,173) .. controls (241,169) and (232,118) .. (263,110) ;
%Curve Lines [id:da6207743970815653]
\draw    (284,294) .. controls (247,286) and (229,205) .. (279,194) ;
%Curve Lines [id:da6545930492297507]
\draw    (285,310) .. controls (193,319) and (146,84) .. (266,90) ;
%Straight Lines [id:da5427015714753514]
\draw    (325,71) -- (325.5,97.5) ;
%Straight Lines [id:da0765883210924918]
\draw    (325.5,97.5) -- (327,137) ;
%Straight Lines [id:da7475113857871238]
\draw    (326,187) -- (326,217) ;
%Straight Lines [id:da9711425850166147]
\draw    (326,157) -- (326,187) ;
%Straight Lines [id:da0824378975694724]
\draw    (332,265) -- (332,302) ;
%Straight Lines [id:da8019729034314078]
\draw    (332,302) -- (332,336) ;
%Curve Lines [id:da704202797886671]
\draw    (325,71) .. controls (316,35) and (266,86) .. (295,124) ;
%Curve Lines [id:da19025649540430978]
\draw    (295,124) .. controls (307,138) and (319,144) .. (326,157) ;
%Curve Lines [id:da1892550022063686]
\draw    (403,339) .. controls (443,309) and (342,281) .. (326,217) ;
%Curve Lines [id:da5373811733152241]
\draw    (403,339) .. controls (396,348) and (338,376) .. (332,336) ;

% Text Node
\draw (375,90) node [anchor=north west][inner sep=0.75pt]   [align=left] {$1''$};
% Text Node
\draw (262,91) node [anchor=north west][inner sep=0.75pt]   [align=left] {$4''$};
% Text Node
\draw (278,176) node [anchor=north west][inner sep=0.75pt]   [align=left] {$4$};
% Text Node
\draw (362,183) node [anchor=north west][inner sep=0.75pt]   [align=left] {$1$};
% Text Node
\draw (278,293) node [anchor=north west][inner sep=0.75pt]   [align=left] {$4'$};
% Text Node
\draw (377,292) node [anchor=north west][inner sep=0.75pt]   [align=left] {$1'$};
% Text Node
\draw (325,65) node [anchor=north west][inner sep=0.75pt]   [align=left] {$5'$};
% Text Node
\draw (328,114) node [anchor=north west][inner sep=0.75pt]   [align=left] {$3$};
% Text Node
\draw (314,154) node [anchor=north west][inner sep=0.75pt]   [align=left] {$5$};
% Text Node
\draw (315,197) node [anchor=north west][inner sep=0.75pt]   [align=left] {$2$};
% Text Node
\draw (321,276) node [anchor=north west][inner sep=0.75pt]   [align=left] {$6$};
% Text Node
\draw (317,312) node [anchor=north west][inner sep=0.75pt]   [align=left] {$2'$};
% Text Node
\draw (215,176) node [anchor=north west][inner sep=0.75pt]   [align=left] {$P(4)$};
% Text Node
\draw (416,177) node [anchor=north west][inner sep=0.75pt]   [align=left] {$P(1)$};

\end{tikzpicture}
\end{center}

\vspace{0.5cm}
\begin{center}
\tikzset{every picture/.style={line width=0.75pt}}
\begin{tikzpicture}[x=18pt,y=18pt,yscale=1,xscale=1]
\fill (0,0) circle (0.5ex);
\fill (5,0) circle (0.5ex);
\fill (0.3,1.7) circle (0.1ex);
\fill (0,1.9) circle (0.1ex);
\fill (-0.3,2.1) circle (0.1ex);
\fill (4.8,1.8) circle (0.1ex);
\fill (4.5,1.9) circle (0.1ex);
\fill (4.2,2) circle (0.1ex);
\fill (0,-1.5) circle (0.1ex);
\fill (-0.3,-1.5) circle (0.1ex);
\fill (-0.6,-1.5) circle (0.1ex);
\fill (5,-1.3) circle (0.1ex);
\fill (4.7,-1.3) circle (0.1ex);
\fill (4.4,-1.3) circle (0.1ex);
\node at(0,1) {$0$};
\node at(4.8,1.2) {$0'$};
\node at(1,1.1) {$1$};
\node at(6,1.4) {$2$};
\node at(1.5,0.3) {$3$};
\node at(6,0.3) {$3'$};
\node at(1.2,-0.4) {$4$};
\node at(6.2,-0.5) {$5$};
\node at(1.1,-1.2) {$6$};
\node at(6.2,-1.2) {$6'$};
\node at(0.6,-1.7) {$7$};
\node at(5.2,-1.7) {$8$};
\draw    (0,0) .. controls (1,4) and (6,4) .. (5,0) ;
\draw    (0,0) -- (1.5,1.5) ;
\draw    (5,0) -- (6.5,1.5) ;
\draw    (0,0) .. controls (4,1) and (9,1) .. (5,0) ;
\draw    (0,0) -- (2,-0.6) ;
\draw    (5,0) -- (7,-0.6) ;
\draw    (0,0) .. controls (3,-3) and (8,-3) .. (5,0) ;
\draw    (0,0) -- (0.5,-2) ;
\draw    (5,0) -- (5.5,-2) ;
\end{tikzpicture}
\end{center}

\begin{Def}\label{standard sequence}
Let $E$ be an f-BC.
\begin{itemize}
\item[(1)] Let $n$ be a positive integer. We call $p=(e_{n},\cdots,e_2,e_1)$ a sequence of angles of length $n$ in $E$, if $e_i'$s are angles in $E$ and $P(g\cdot e_i)=P(e_{i+1})$ for all $1\leq i\leq n-1$. Moreover, for every $e\in E$, we call $()_e$ a sequence of angles of length $0$ in $E$ at $e$, or a trivial sequence of $E$ at $e$.
\item[(2)] A sequence of the form $p=(g^{n-1}\cdot e,\cdots,g\cdot e,e)$ with $e\in E$ and $0\leq n\leq d(e)$ is called a standard sequence  (associated to the vertex $G\cdot e$) of $E$ (we define $p=()_e$ when $n=0$). For $n>0$, the source (resp. terminal) of standard sequence $p=(g^{n-1}\cdot e,\cdots,g\cdot e,e)$ is defined to be $e$ (resp. $g^{n}\cdot e$); both the source and the terminal of the trivial sequence $()_e$ are defined to be $e$. The composition of two standard sequences is defined in a natural way.
\item[(3)] A standard sequence of the form $(g^{d(e)-1}\cdot e,\cdots,g\cdot e,e)$ with $e\in E$ is called a full sequence of $E$.
\end{itemize}
\end{Def}

 For a standard sequence $p=(g^{n-1}\cdot e,\cdots,g\cdot e,e)$, we can define two associated standard sequences (which can be called the left complement and the right complement of $p$ respectively) $$\prescript{\wedge}{}{p}=\begin{cases}
(g^{d(e)-1}\cdot e,\cdots,g^{n+1}\cdot e,g^{n}\cdot e), \text{ if } 0<n<d(e); \\
()_{g^{d(e)}(e)}, \text{ if } n=d(e); \\
(g^{d(e)-1}\cdot e,\cdots,g\cdot e,e), \text{ if } n=0 \text{ and } p=()_e,
\end{cases}$$ and
$$p^{\wedge}=\begin{cases}
(g^{-1}\cdot e,g^{-2}\cdot e,\cdots,g^{n-d(e)}\cdot e), \text{ if } 0<n<d(e); \\
()_{e}, \text{ if } n=d(e); \\
(g^{-1}\cdot e,\cdots,g^{-d(e)}\cdot e), \text{ if } n=0 \text{ and } p=()_e.
\end{cases}$$
Note that for a standard sequence $p$ of $E$, both $\prescript{\wedge}{}{p}p$ and $pp^{\wedge}$ are full sequences of $E$.

For a sequence $p=(e_{n},\cdots,e_2,e_1)$ in $E$, we can associate a formal sequence as follows:
$$L(p)=\begin{cases}
L(e_n)\cdots L(e_2)L(e_1), \text{ if } p=(e_{n},\cdots,e_2,e_1)\text{ is a sequence of length}>0; \\
1_{P(e)}, \text{ if } p=()_e\text{ is a trivial sequence at }e.
\end{cases}$$
Moreover, for a set $\mathscr{X}$ of sequences, define $L(\mathscr{X})=\{L(p)\mid p\in \mathscr{X}\}$.

\begin{Def}\label{parallel standard sequence}
Let $E$ be an f-BC, $p$, $q$ be two sequences of angles in $E$. Denote $p\equiv q$ if the associated formal sequences $L(p)$ and $L(q)$ are equal. In this case we say that $p$, $q$ are identical.
\end{Def}

\begin{Rem1}\label{identical}
\begin{itemize}
\item[(1)] Any standard sequence which is identical to a trivial sequence $()_{e}$ is also a trivial sequence $()_{e'}$ such that $P(e)=P(e')$.

\item[(2)] By $(f6)$ we see that any standard sequence which is identical to a full sequence $p$ of the form $(g^{d(e)-1}\cdot e,\cdots,g\cdot e,e)$ is also a full sequence $q$ of the form $(g^{d(h)-1}\cdot h,\cdots,g\cdot h,h)$ with $P(e)=P(h)$.

\item[(3)] By $(f4)$ and $(f5)$, for standard sequences $p$, $q$, $\prescript{\wedge}{}{p}\equiv\prescript{\wedge}{}{q}$ if and only if $p^{\wedge}\equiv q^{\wedge}$.
\end{itemize}
\end{Rem1}

For a set $\mathscr{X}$ of standard sequences, denote $\prescript{\wedge}{}{\mathscr{X}}=\{\prescript{\wedge}{}{p}\mid p\in \mathscr{X}\}$ (resp. $\mathscr{X}^{\wedge}=\{p^{\wedge}\mid p\in \mathscr{X}\}$), and denote $[\mathscr{X}]=\{$ standard sequence $q\mid q$ is identical to some $p\in\mathscr{X}\}$.

\begin{Def}\label{f-BC-type-S}
An f-BC $E$ is said to be of type S (or $E$ is an $f_s$-BC in short) if it satisfies additionally the following condition:

$(f7)$ For standard sequences $p\equiv q$, $[[\prescript{\wedge}{}{p}]^{\wedge}]=[[\prescript{\wedge}{}{q}]^{\wedge}]$ or $[\prescript{\wedge}{}{[p^{\wedge}]}]=[\prescript{\wedge}{}{[q^{\wedge}]}]$.
\end{Def}

In particular, if $E$ is an $f_s$-BC such that each polygon of $E$ has exactly two elements, then we call $E$ a fractional Brauer graph of type S (abbr. $f_s$-BG).

\begin{Rem1}\label{two-conditions-are equivalent}
\begin{itemize}
\item[(1)] For a standard sequence $p$, we have $[\prescript{\wedge}{}{[p^{\wedge}]}]=[\prescript{\wedge}{}{[(\prescript{\wedge}{}{p})^{\wedge\wedge}]}]
=[\prescript{\wedge}{}{([\prescript{\wedge}{}{p}]^{\wedge\wedge})}]=[[\prescript{\wedge}{}{p}]^{\wedge}]$. It follows that the conditions $[[\prescript{\wedge}{}{p}]^{\wedge}]=[[\prescript{\wedge}{}{q}]^{\wedge}]$ and $[\prescript{\wedge}{}{[p^{\wedge}]}]=[\prescript{\wedge}{}{[q^{\wedge}]}]$ in Definition \ref{f-BC-type-S} are equivalent.
\item[(2)] For standard sequences $p\equiv q$, the equation $[[\prescript{\wedge}{}{p}]^{\wedge}]=[[\prescript{\wedge}{}{q}]^{\wedge}]$ is always true if $p$ is a trivial or a full sequence.
\end{itemize}
\end{Rem1}

In Example \ref{example0}, the partition $L$ is trivial. So for nontrivial standard sequences $p,q$ of $E$, $p\equiv q$ implies $p=q$. Then $(f7)$ holds and $E$ is an $f_s$-BC and indeed an $f_s$-BG.

In Example \ref{example1}, the only case for $p\equiv q$ and $p\neq q$ for nontrivial standard sequences $p$, $q$ is $p=(1)$ and $q=(1^{'})$. Since $[[\prescript{\wedge}{}{p}]^{\wedge}]=[[\prescript{\wedge}{}{q}]^{\wedge}]=\{(1),(1')\}$, $E$ is an $f_s$-BC and indeed an $f_s$-BG.

In Example \ref{example2}, let $p=(a_1)$ and $q=(a_2)$ be identical standard sequences in $E$, then $[[\prescript{\wedge}{}{p}]^{\wedge}]=\{(a_1),(a_2),(a_3,a_2,a_3)\}$ and $[[\prescript{\wedge}{}{q}]^{\wedge}]=\{(a_1),(a_2)\}$. Therefore $E$ is an f-BC but not an $f_s$-BC.

In Example \ref{example3}, let $p=(2,1)$ and $q=(2',1')$ be identical standard sequences in $E$, then $[[\prescript{\wedge}{}{p}]^{\wedge}]=\{(2,1),(2',1'),(3,1'')\}$ and $[[\prescript{\wedge}{}{q}]^{\wedge}]=\{(2,1),(2',1')\}$. Therefore $E$ is an f-BC but not an $f_s$-BC.

In Example \ref{example4}, the partition $L$ is trivial. Then $(f7)$ holds and $E$ is an $f_s$-BC.

\begin{Def}\label{f-BC-type-MS}
An f-BC $E=(E,P,L,d)$ is said to be of type MS (or $E$ is an $f_{ms}$-BC in short) if partition $L$ of $E$ is trivial. If $E$ is an f-BC of type MS such that $P(e)$ contains exactly two elements for each $e\in E$, then $E$ is said to be a fractional Brauer graph of type MS (abbr. $f_{ms}$-BG).
\end{Def}

Clearly every f-BC of type MS is of type S, and Example \ref{example0} and Example \ref{example4} are f-BCs of type MS.

We summarize our discussion by the following diagram:

$$\xymatrix@R=3px{
BC \ar@{=>}[r] & f_{ms}\mbox{-}BC \ar@{=>}[r] & f_{s}\mbox{-}BC  \ar@{=>}[r] & f\mbox{-}BC \\
&&& \\
BG \ar@{=>}[uu]\ar@{=>}[r] & f_{ms}\mbox{-}BG \ar@{=>}[uu]\ar@{=>}[r] & f_{s}\mbox{-}BG \ar@{=>}[uu] & \\
}$$

\section{Fractional Brauer configuration category}

In this section, we will associate every f-BC $E$ a $k$-category $\Lambda_E$ which we call a fractional Brauer configuration category. It turns out that $\Lambda_E$ is a locally bounded $k$-category in the sense of Bongartz and Gabriel.

\subsection{The quiver of a fractional Brauer configuration category}
\

\begin{Def}\label{f-BC-category}
For an f-BC $E=(E,P,L,d)$, we associate a quiver $Q_{E}$ as follows.
\
\begin{itemize}
    \item The set $(Q_{E})_{0}$ of vertices is given by $\{P(e)\mid e\in E\}$;
    \item The set $(Q_{E})_{1}$ of arrows is given by $\{L(e)\mid e\in E\}$, where the arrow $L(e)$ has the source $P(e)$ and the terminal $P(g\cdot e)$.
\end{itemize}
Denote $s: (Q_{E})_{1}\rightarrow (Q_{E})_{0}$ and $t: (Q_{E})_{1}\rightarrow (Q_{E})_{0}$ by two maps with $s$ sending each arrow to its source and $t$ sending each arrow to its terminal.
\end{Def}

\begin{Rem1}\label{interpretation-in-terms-of-f1-f6} We have the following interpretations of $Q_{E}$ in terms of the conditions of f-BC.
\begin{itemize}
    \item[(1)] $(f1)$ and $(f2)$ ensure that the definition of the source and the terminal of an arrow $L(e)$ of $Q_E$ are independent to the choice of the representative $e$ of the class $L(e)$.
    \item[(2)] Partition $L$ means that the arrows of $Q_{E}$ are obtained by identifying $L(e)$ with $L(e')$ in $Q_{E}$ if $e$ and $e'$ belong to the same partition class of $L$. Moreover, since $P(e)$ (hence also $L(e)$) is a finite set for each $e$, the quiver $Q_E$ is locally finite, that is, the number of arrows starting or stopping at any vertex is finite.
    \item[(3)] $(f4)$ (resp. $(f5)$) ensures that the Nakayama automorphism $\sigma$ of $E$ induces a permutation on $(Q_{E})_{0}$ (resp. $(Q_{E})_{1}$). $(f4)$ and $(f5)$ together with $(f3)$ ensure that $\sigma$ commutes with $s$ and $t$ so that it induces an automorphism of the quiver $Q_E$. We also denote this automorphism by $\sigma$.
    \item[(4)] $p=(e_{n},\cdots,e_2,e_1)$ is a sequence of $E$ if and only if $L(p)=L(e_n)\cdots L(e_2)L(e_1)$ is a path of $Q_{E}$, and we write a path of $Q_E$ from right to left. For a trivial sequence $p=()_e$, $L(p)=1_{P(e)}$ is the trivial path of $Q_E$ at vertex $P(e)$.

    \item[$(5)$] Every finite (resp. infinite) $\langle g \rangle$-orbit $G\cdot e$ corresponds to a cycle $L(g^{\mid G\cdot e\mid -1}\cdot e)\cdots L(g\cdot e)L(e)$ of length $\mid G\cdot e\mid$ (resp. an infinite path $\cdots L(g\cdot e)L(e)L(g^{-1}\cdot e)\cdots$) in the quiver $Q_E$.
\end{itemize}
\end{Rem1}

We usually denote a polygon in an f-BC $E$ by $P(i)=\{i,i',i'',\cdots\}$ and denote the corresponding vertex in $Q_E$ by $i$. In Example \ref{example0}, $Q_{E}$ is the following quiver

\begin{center}
\tikzset{every picture/.style={line width=0.75pt}}
\begin{tikzpicture}[x=22pt,y=18pt,yscale=1,xscale=1]
\node at(0,0) {$1$};
\node at(5,0) {$2$};\node at(5.5,0) {.};
\node at(2.5,4.33) {$3$};
\node at(2.5,0.6) {$L(1)$};
\node at(2.5,-0.6) {$L(1')$};
\node at(3.2,1.9) {$L(2)$};
\node at(4.5,2.8) {$L(2')$};
\node at(1.8,1.9) {$L(3)$};
\node at(0.5,2.7) {$L(3')$};
\draw[->]    (0.5,0.2)--(4.5,0.2) ;
\draw[->]    (0.5,-0.2)--(4.5,-0.2) ;
\draw[->]    (4.7,0.3)--(2.7,4) ;
\draw[->]    (4.9,0.6)--(2.9,4.3) ;
\draw[->]    (2.3,4)--(0.3,0.3) ;
\draw[->]    (2.1,4.3)--(0.1,0.6) ;
\end{tikzpicture}
\end{center}

In Example \ref{example1}, $Q_{E}$ is the following quiver

$$
\vcenter{
\xymatrix {
	&&&&1\ar[dd]^{L(1)} \\
    4\ar[urrrr]|{L(4')}\ar@(dl,ul)^{L(4)}&&3\ar[urr]|{L(3)}\ar@(dl,ul)^{L(3')}&& \\
    &&&&2\ar[ullll]|{L(2')}\ar[ull]|{L(2)}
}}.$$ \\
\medskip

In Example \ref{example2}, $Q_{E}$ is the following quiver

$$
\vcenter{
\xymatrix {
	a\ar@(dl,ul)^{L(a_1)}\ar@(ur,dr)^{L(a_3)}
}}.$$ \\
\medskip

In Example \ref{example3}, $Q_{E}$ is the following quiver

$$
\vcenter{
\xymatrix {
	&&&&3\ar[lddd]^{L(3)} \\
    &1\ar[rr]_{L(1)}\ar[urrr]^{L(1'')}&&2\ar[dd]_{L(2)}& \\
    &&&& \\
    &5\ar[uu]_{L(5)}&&4\ar[ll]_{L(4)}\ar[llld]^{L(4')}& \\
    6\ar[uuur]^{L(6)}&&&&
}}.$$
\medskip

In Example \ref{example4}, $Q_{E}$ is the following quiver

\begin{center}
\tikzset{every picture/.style={line width=0.75pt}}
\begin{tikzpicture}[x=50pt,y=30pt,yscale=1,xscale=1]
\node at(0,0) {$0$};
\node at(1,1) {$1$};
\node at(1,-1) {$2$};
\node at(2,0) {$3$};
\node at(3,1) {$5$};
\node at(3,-1) {$4$};
\node at(4,0) {$6$};
\node at(5,1) {$7$};
\node at(5,-1) {$8$};
\node at(0.33,0.7) {$L(0)$};
\node at(0.33,-0.7) {$L(0')$};
\node at(1.67,0.7) {$L(1)$};
\node at(1.67,-0.7) {$L(2)$};
\node at(2.33,0.7) {$L(3')$};
\node at(2.33,-0.7) {$L(3)$};
\node at(3.67,0.7) {$L(5)$};
\node at(3.67,-0.7) {$L(4)$};
\node at(4.33,0.7) {$L(6)$};
\node at(4.33,-0.7) {$L(6')$};
\fill (-0.4,0) circle (0.1ex);
\fill (-0.5,0) circle (0.1ex);
\fill (-0.6,0) circle (0.1ex);
\fill (5.4,0) circle (0.1ex);
\fill (5.5,0) circle (0.1ex);
\fill (5.6,0) circle (0.1ex);
\draw[->]    (0.2,0.2)--(0.8,0.8) ;
\draw[->]    (0.2,-0.2)--(0.8,-0.8) ;
\draw[->]    (1.2,0.8)--(1.8,0.2) ;
\draw[->]    (1.2,-0.8)--(1.8,-0.2) ;
\draw[->]    (2.2,0.2)--(2.8,0.8) ;
\draw[->]    (2.2,-0.2)--(2.8,-0.8) ;
\draw[->]    (3.2,0.8)--(3.8,0.2) ;
\draw[->]    (3.2,-0.8)--(3.8,-0.2) ;
\draw[->]    (4.2,0.2)--(4.8,0.8) ;
\draw[->]    (4.2,-0.2)--(4.8,-0.8) ;
\end{tikzpicture}.
\end{center}

\begin{Lem}\label{interpretation-of-(fR1)}
Let $E=(E,P,L,d)$ be an f-BC. For a path $p=L(e_n)\cdots L(e_2)L(e_1)$ of $Q_E$, the following conditions are equivalent:
\begin{itemize}
    \item [(1)] There exists some $e\in E$ such that $p=L(g^{n-1}\cdot e)\cdots L(g\cdot e)L(e)$.
    \item [(2)] $\bigcap_{i=1}^n g^{n-i}\cdot L(e_i)\neq\emptyset$.
\end{itemize}
\end{Lem}

\begin{proof}
If there exists some $e\in E$ such that $p=L(g^{n-1}\cdot e)\cdots L(g\cdot e)L(e)$, then $g^{n-i}\cdot L(e_i)=g^{n-i}\cdot L(g^{i-1}\cdot e)$ for each $1\leq i\leq n$ and $g^{n-1}\cdot e\in\bigcap_{i=1}^n g^{n-i}\cdot L(e_i)$. Therefore $\bigcap_{i=1}^n g^{n-i}\cdot L(e_i)\neq\emptyset$. Conversely, if $\bigcap_{i=1}^n g^{n-i}\cdot L(e_i)\neq\emptyset$, let $h\in\bigcap_{i=1}^n g^{n-i}\cdot L(e_i)$ and let $e=g^{1-n}\cdot h$. Since $e\in g^{1-i}\cdot L(e_i)$ for each $1\leq i\leq n$, $L(g^{i-1}\cdot e)=L(e_i)$ for each $1\leq i\leq n$. Therefore $p=L(g^{n-1}\cdot e)\cdots L(g\cdot e)L(e)$.
\end{proof}

The above quiver $Q_E$ defines a path category $kQ_E$ whose objects are the vertices of $Q_E$ and whose morphisms are generated by the paths of $Q_E$.

\subsection{The relations of a fractional Brauer configuration category}
\

\begin{Def} \label{ideal-I_E}
For an f-BC $E=(E,P,L,d)$, we define an ideal $I_{E}$ of the path category $kQ_E$ which is generated by the following three types of relations:
\begin{itemize}
\item[$(fR1)$] $L(g^{d(e)-1-k}\cdot e)\cdots L(g\cdot e)L(e)-L(g^{d(h)-1-k}\cdot h)\cdots L(g\cdot h)L(h)$, where $k\geq 0$, $P(e)=P(h)$ and $L(g^{d(e)-i}\cdot e)=L(g^{d(h)-i}\cdot h)$ for $1\leq i\leq k$.
\item[$(fR2)$] Paths of the form $L(e_n)\cdots L(e_2)L(e_1)$ with $\bigcap_{i=1}^n g^{n-i}\cdot L(e_i)=\emptyset$ for $n>1$.
\item[$(fR3)$] Paths of the form $L(g^{n-1}\cdot e)\cdots L(g\cdot e)L(e)$ for $n>d(e)$.
\end{itemize}
We call the quotient category $\Lambda_{E}=kQ_{E}/I_{E}$ the fractional Brauer configuration category (abbr. f-BCC) of $E$. Moreover, if $E$ is an $f_s$-BC ( resp. $f_{s}$-BG, $f_{ms}$-BC, $f_{ms}$-BG), then we call $\Lambda_E$ an $f_s$-BCC ( resp. $f_{s}$-BGC, $f_{ms}$-BCC, $f_{ms}$-BGC).
\end{Def}

\begin{Rem1}\label{interpretations-to-the-ideal-$I_E$} We have the following interpretations to the ideal $I_E$:
\begin{itemize}
    \item $(fR1)$ means that $L(p)-L(q)$ is contained in $I_E$ if $p,q$ are standard sequences of $E$ with $\prescript{\wedge}{}{p}\equiv \prescript{\wedge}{}{q}$. Note that $(fR1)$ is equal to $(R1)$ in BCA case, but if the partition $L$ is nontrivial, then $(fR1)$ contains more elements in general.
    \item $(fR2)$ means that each path of $Q_E$ not of the form $L(g^{n-1}\cdot e)\cdots L(g\cdot e)L(e)$ is contained in $I_E$ (see Lemma \ref{interpretation-of-(fR1)}). Unfortunately, Conditions $(f1)-(f6)$ do not ensure that all full sequences define nonzero paths in $kQ_E/I_E$ (see Examples \ref{example2}, \ref{example3}), for this to be true, we need the condition $(f7)$ (see Remark \ref{full-squences-are-nonzero}). Note that $(fR2)$ is equal to $(R2)$ in BCA case, but if the partition $L$ is nontrivial, then $(fR2)$ contains more elements in general.
    \item $(fR3)$ defines the maximal possible paths in $kQ_E$ which are not contained in $I_E$. Together with $(f3)$ in Definition \ref{f-BC} it is not hard to see that, for each vertex $x=P(e)\in (Q_E)_0$, there is a natural number $N_x=(\mathrm{max}\{d(h)\mid h\in P(e)\}+1)$ such that $I_E$ contains each path of length $\geq N_x$ which starts or stops at $x$. Note that when the partition $L$ is trivial and there is no $1$-gon, then $(fR3)$ can be removed, and $(fR2)$ can be simplified to the case $n=2$.
    \item Condition $(f6)$ in Definition \ref{f-BC} ensures that $I_E$ is contained in the ideal of $kQ_E$ generated by arrows. Otherwise, there exists a relation $r=L(g^{d(e)-1-k}\cdot e)\cdots L(g\cdot e)L(e)-L(g^{d(h)-1-k}\cdot h)\cdots L(g\cdot h)L(h)$ of type $(fR1)$ such that the length of one of the paths $L(g^{d(e)-1-k}\cdot e)\cdots L(g\cdot e)L(e)$ and $L(g^{d(h)-1-k}\cdot h)\cdots L(g\cdot h)L(h)$ is $0$. If the length of $L(g^{d(e)-1-k}\cdot e)\cdots L(g\cdot e)L(e)$ is $0$, then $k=d(e)$. Since $r$ is nonzero, $k<d(h)$. So $L(g^{d(e)-1}\cdot e)\cdots L(g\cdot e)L(e)$ is a proper subsequence of $L(g^{d(h)-1}\cdot h)\cdots L(g\cdot h)L(h)$, which contradicts the condition $(f6)$.
\end{itemize}

\end{Rem1}

\begin{Rem1}\label{two-quivers-with-relations-are-identical}
For a BG $\Gamma$, if we consider $\Gamma$ as an finite $f_{ms}$-BC $E$ with integral f-degree, then the quiver with relations $(Q_E,I_E)$ is just the quiver with relations $(Q_{\Gamma},I_{\Gamma})$ defined in Section 2.
\end{Rem1}

In Example \ref{example0}, $I_{E}$ is generated by the following relations: $$ L(2)L(1)-L(2')L(1'), L(3)L(2)-L(3')L(2'), L(1)L(3)-L(1')L(3'), $$
$$ L(2')L(1), L(2)L(1'), L(3')L(2), L(3)L(2'), L(1')L(3), L(1)L(3').$$

\medskip
In Example \ref{example1}, $I_{E}$ is generated by the following relations: $$L(3')-L(2)L(1)L(3), L(4)-L(2')L(1)L(4'), L(3)L(2)-L(4')L(2'), L(3')L(2),$$
$$L(3)L(3'), L(4)L(2'), L(4')L(4), L(2)L(1)L(4'), L(2')L(1)L(3), L(1)L(3)L(2)L(1).$$

\medskip

In Example \ref{example2}, let $p=(a_1)$, $q=(a_3,a_2,a_3)$ be standard sequences, then $\prescript{\wedge}{}{p}\equiv\prescript{\wedge}{}{q}$. Therefore $L(a_1)=L(a_3)L(a_1)L(a_3)$ in $\Lambda_{E}$. Similarly, $L(a_3)=L(a_1)L(a_3)L(a_1)$ in $\Lambda_{E}$. We have
\begin{equation} L(a_1)=L(a_3)L(a_1)L(a_3)=L(a_3)L(a_3)L(a_1)L(a_3)L(a_3)
\end{equation}
in $\Lambda_{E}$. Since $$L(a_3)\cap g\cdot L(a_3)\cap g^{2}\cdot L(a_1)=\{a_3,a_4\}\cap\{a_2,a_4\}\cap\{a_1,a_2\}=\emptyset,$$
$L(a_3)L(a_3)L(a_1)=0$ in $\Lambda_{E}$. Therefore $L(a_1)=0$ in $\Lambda_{E}$. Similarly, $L(a_3)=0$ in $\Lambda_{E}$. So $I_{E}$ is generated by $L(a_1)$, $L(a_3)$ and $\Lambda_{E}\cong k$.
\medskip

In Example \ref{example3}, $I_{E}$ is generated by the following relations: $$L(5)L(4)-L(6)L(4'), L(2)L(1)-L(3)L(1''), L(4')L(3), L(1'')L(6).$$

In Example \ref{example4}, $I_{E}$ is generated by the following relations: $$L(3i+1)L(3i)-L(3i+2)L(3i'), L(3i')L(3i-2), L(3i)L(3i-1),$$ where $i\in\mathbb{Z}$.

\medskip
Recall from \cite[Definition 2.1]{BG} that a locally bounded category is a $k$-category $\Lambda$ satisfying the following three conditions:
\begin{itemize}
    \item For each $x\in \Lambda$, the endomorphism algebra $\Lambda(x,x)$ is local.
    \item Distinct objects of $\Lambda$ are not isomorphic.
    \item For each $x\in \Lambda$, $\sum_{y\in \Lambda}\mathrm{dim}_{k}\Lambda(x,y)< \infty$ and $\sum_{y\in \Lambda}\mathrm{dim}_{k}\Lambda(y,x)< \infty$.
\end{itemize}

For the definition of the radical of a locally bounded category, we refer to \cite[Section 2]{BG}.

\begin{Thm}\label{f-BC algebra is locally bounded}
Let $E=(E,P,L,d)$ be an f-BC, and let $\Lambda_{E}=kQ_{E}/I_{E}$ be the associated f-BCC. Then we have the following.
\begin{itemize}
    \item[(1)] Let $J$ be the ideal of $\Lambda_{E}$ generated by the arrows of $Q_E$. Then $J$ is the radical of $\Lambda_{E}$.
    \item[(2)] $\Lambda_{E}$ is locally bounded $k$-category.
    \item[(3)] The Nakayama automorphism $\sigma$ of $E$ induces an automorphism of the category $\Lambda_E$ which is also denoted by $\sigma$.
\end{itemize}
\end{Thm}

\begin{proof}
For $e\in E$, each nonzero path in $\Lambda_{E}$ starting at $P(e)$ is of the form $L(g^{n-1}\cdot h)\cdots L(g\cdot h)L(h)$, where $h\in P(e)$ and $0\leq n\leq d(h)$. Since $\bigoplus_{a\in (Q_{E})_{0}}^{}\Lambda_{E}(P(e),a)$ is generated by such paths as a $k$-space, and since $P(e)$ is finite, $\mathrm{dim}_{k}(\bigoplus_{a\in (Q_{E})_{0}}^{}\Lambda_{E}(P(e),a))<\infty$. Similarly, $\mathrm{dim}_{k}(\bigoplus_{a\in (Q_{E})_{0}}^{}\Lambda_{E}(a,P(e)))<\infty$.

For each morphism $f:x\rightarrow y$ in $J$, to show that $f$ is in the radical of $\Lambda_{E}$, we need to show that ${id}_{x}-f'f$ is invertible for all $f'\in \Lambda_{E}(y,x)$. Let $x=P(e)$. Then each path starting at $x$ of length larger than $\mathrm{max}\{d(h)\mid h\in P(e)\}$ is zero in $\Lambda_{E}$. Therefore $f'f$ is nilpotent and ${id}_{x}-f'f$ is invertible. Conversely, for each morphism $f:x\rightarrow y$ with $f\notin J$, there exists $\lambda\in k^{*}$ such that $f=\lambda{id}_{x}+f'$ with $f'\in J$. Since $f'$ is nilpotent, $f$ is invertible. Therefore $f$ is not in the radical of $\Lambda_{E}$. This proves (1).

For each object $x$ in $\Lambda_{E}$, $\Lambda(x,x)$ is a finite dimensional algebra. Denotes $J'$ the ideal in $kQ_E$ generated by arrows, then $(f6)$ ensures that $I_E$ is contained in $J'$ (see Remark \ref{interpretations-to-the-ideal-$I_E$}), and we have $J=J'/I_E$. Then $\Lambda(x,x)/(\mathrm{rad}\Lambda(x,x))=\Lambda(x,x)/J(x,x)\cong kQ(x,x)/J'(x,x)\cong k$. Therefore $\Lambda(x,x)$ is local. For different objects $x$, $y$ in $\Lambda_{E}$, $\Lambda(x,y)=J(x,y)$ is contained in the radical of $\Lambda_{E}$, thus each morphism $f:x\rightarrow y$ is not an isomorphism, so $x\ncong y$. Therefore $\Lambda_{E}$ is locally bounded. This proves (2).

By Remark \ref{interpretation-in-terms-of-f1-f6} (3), $\sigma$ induces an automorphism of the quiver $Q_E$. Moreover, it is straightforward to show that the automorphism $\sigma$ of $Q_E$ also induces an automorphism of the category $\Lambda_E$. This proves (3).
\end{proof}

\subsection{The category $\Lambda_E$ in type S is a locally bounded Frobenius category}
\

\begin{Def}\label{Frobenius category}
We call a locally bounded category $\Lambda$ to be Frobenius if for every object $x$ of $\Lambda$, there exists some objects $y,z$ of $\Lambda$ such that $\Lambda(-,x)\cong D\Lambda(y,-)$ and $\Lambda(x,-)\cong D\Lambda(-,z)$, where $D$ denotes the usual $k$-duality on vector spaces.
\end{Def}

Note that such objects $y,z$ in the definition above are also unique by Yoneda's lemma. Denote by mod$\Lambda$ the category of finitely generated $\Lambda$-modules (here, a finitely generated $\Lambda$-module means it is isomorphic to a quotient of a finite direct sum of representable contravariant functors from $\Lambda$ to mod$k$). According to \cite[Section 2]{BG}, the indecomposable projective modules in mod$\Lambda$ are isomorphic to $\Lambda(-,x)$ ($x\in \Lambda$) and the indecomposable injective modules in mod$\Lambda$ are isomorphic to $D\Lambda(y,-)$ ($y\in \Lambda$). Therefore if $\Lambda$ is a locally bounded Frobenius category, then the category mod$\Lambda$ is a Frobenius category in the sense of Happel \cite{H}. This demonstrates our terminology of locally bounded Frobenius category.

\begin{Def}\label{relation R}
Let $E=(E,P,L,d)$ be an f-BC and let $$\mathscr{E}=\{L(p)\mid p\mbox{ is a standard sequence of }E\},$$ which is a set of paths of $Q_E$. Define a relation $R$ on $\mathscr{E}$ as follows: for $u$, $v\in\mathscr{E}$, $uRv$ if and only if there exist some standard sequences $p$, $q$ of $E$ such that $u=L(p)$, $v=L(q)$ and $\prescript{\wedge}{}{p}\equiv \prescript{\wedge}{}{q}$.
\end{Def}

\begin{Rem1}
\begin{itemize}
\item[(1)] If $u$, $v\in\mathscr{E}$ and $uRv$, then $u,v$ have the same source and the same terminal.
\item[(2)] For $u$, $v\in\mathscr{E}$ with $uRv$, if $u=L(p)$ such that $p$ is a trivial sequence $()_e$, then $v=L(q)$ for some trivial sequence $q=()_{e'}$ such that $P(e)=P(e')$, and therefore $u=v$.
\item[(3)] For paths $u,v$ of $\mathscr{E}$ of length $\geq 1$, $uRv$ if and only if $u-v$ is a relation of $I_E$ of type $(fR1)$.
\end{itemize}
\end{Rem1}

\begin{Lem}\label{independence}
Let $E=(E,P,L,d)$ be an f-BC and let $$L(g^{n-1}\cdot e)\cdots L(g\cdot e)L(e)=L(g^{n-1}\cdot h)\cdots L(g\cdot h)L(h)$$ be a path of $Q_{E}$. Then $n>d(e)$ (resp. $n<d(e)$) if and only if $n>d(h)$ (resp. $n<d(h)$).
\end{Lem}

\begin{proof}
Suppose $n>d(e)$ and $n\leq d(h)$. Then $L(g^{d(e)-1}\cdot e)\cdots L(g\cdot e)L(e)$ is a proper subsequence of $L(g^{d(h)-1}\cdot h)\cdots L(g\cdot h)L(h)$, which contradicts to $(f6)$. Therefore $n>d(e)$ if and only if $n>d(h)$. Similarly $n<d(e)$ if and only if $n<d(h)$.
\end{proof}

\begin{Rem1}\label{disjoint-union}
Let $\mathscr{B}_1$ (resp. $\mathscr{B}_2$) be the set of paths of $Q_E$ which are relations of type $(fR2)$ (resp. of type $(fR3)$) in $I_E$. By Lemma \ref{independence}, we know that the set of paths in $Q_E$ is a disjoint union of $\mathscr{B}_1$, $\mathscr{B}_2$ and $\mathscr{E}$.
\end{Rem1}

\begin{Lem}\label{the-condition-equivalent-to-(f7)}
For an f-BC $E=(E,P,L,d)$, $(f7)$ is equivalent to the following condition:

\noindent $(f7')$ For any standard sequence $p$ of $E$, the set of paths $v\in \mathscr{E}$ with $vRL(p)$ is $\{L(p')\mid\prescript{\wedge}{}{p}\equiv \prescript{\wedge}{}{p'}\}$.
\end{Lem}

\begin{proof}
Suppose that $(f7)$ holds. For a standard sequence $p$ of $E$, if $v$ is a path in $\mathscr{E}$ with $vRL(p)$, then there exists standard sequences $q,r$ of $E$ with $L(p)=L(q)$, $v=L(r)$ and $\prescript{\wedge}{}{q}\equiv \prescript{\wedge}{}{r}$. Then $p\equiv q$ and $r\in[\prescript{\wedge}{}{q}]^{\wedge}\subseteq [[\prescript{\wedge}{}{q}]^{\wedge}]=[[\prescript{\wedge}{}{p}]^{\wedge}]$. So there exists some standard sequence $l$ of $E$ with $\prescript{\wedge}{}{p}\equiv \prescript{\wedge}{}{l}$ and $l\equiv r$, and therefore $v=L(l)$ belongs to the set $\{L(p')\mid\prescript{\wedge}{}{p}\equiv \prescript{\wedge}{}{p'}\}$. Then the condition $(f7')$ holds.

Conversely, suppose that $(f7')$ holds. For standard sequences $p,q$ of $E$ with $p\equiv q$ and for $r\in [[\prescript{\wedge}{}{p}]^{\wedge}]$, there exists some standard sequence $t$ of $E$ with $L(r)=L(t)$ and $\prescript{\wedge}{}{t}\equiv \prescript{\wedge}{}{p}$. Since $L(q)=L(p)$, $L(r)RL(q)$. Then by $(f7')$ we have $L(r)=L(s)$ for some standard sequence $s$ of $E$ with $\prescript{\wedge}{}{q}\equiv \prescript{\wedge}{}{s}$. Therefore $r\equiv s$ and $r\in [[\prescript{\wedge}{}{q}]^{\wedge}]$, which implies that $[[\prescript{\wedge}{}{p}]^{\wedge}]\subseteq [[\prescript{\wedge}{}{q}]^{\wedge}]$. Similarly, we have $[[\prescript{\wedge}{}{q}]^{\wedge}]\subseteq [[\prescript{\wedge}{}{p}]^{\wedge}]$, and the condition $(f7)$ holds.
\end{proof}

\begin{Lem}\label{equivalence relation}
If $E=(E,P,L,d)$ is an $f_s$-BC, that is, $E$ is an f-BC satisfying moreover $(f7)$, then $R$ is an equivalence relation on $\mathscr{E}$.
\end{Lem}

\begin{proof}
By definition, $R$ is reflexive and symmetric. Suppose $uRv$ and $vRw$ for $u$, $v$, $w\in\mathscr{E}$. If $u=L(p)$ for some standard sequence $p$ of $E$, by Lemma \ref{the-condition-equivalent-to-(f7)}, $v=L(q)$ for some standard sequence $q$ of $E$ with $\prescript{\wedge}{}{p}\equiv \prescript{\wedge}{}{q}$, and $w=L(r)$ for some standard sequence $r$ of $E$ with $\prescript{\wedge}{}{q}\equiv \prescript{\wedge}{}{r}$. Therefore $\prescript{\wedge}{}{p}\equiv \prescript{\wedge}{}{r}$ and $uRw$.
\end{proof}

\begin{Lem}\label{key lemma}
Let $E=(E,P,L,d)$ be an $f_s$-BC, $u$, $v$ be paths of $Q_E$ such that $u-v$ is a relation of $I_E$ of type $(fR1)$. For each path $w$ of $Q_E$ whose source (resp. terminal) is equal to the terminal (resp. source) of $u$, one of the following holds: (1) $wu$, $wv$ (resp. $uw$, $vw$) are relations of $I_E$ of type $(fR2)$ or type $(fR3)$; (2) $wu-wv$ (resp. $uw-vw$) is a relation of $I_E$ of type $(fR1)$.
\end{Lem}

\begin{proof}
Since $u-v$ is a relation of $I_E$ of type $(fR1)$, $u$, $v$ are paths of length $\geq 1$ in $\mathscr{E}$ with $uRv$. Let $w$ be a path of $Q_E$ whose source is equals to the terminal of $u$.
We may assume that $l(w)>0$. If $wu$ is neither a relation of $I_E$ of type $(fR2)$ nor a relation of $I_E$ of type $(fR3)$, then $wu\in\mathscr{E}$. Let $wu=L(p)$, where $p$ is a nontrivial standard sequence of $E$. Write $p=p_2 p_1$, where $u=L(p_1)$ and $w=L(p_2)$. Since $l(w), l(u)>0$, both $p_1$ and $p_2$ are nontrivial standard sequences of $E$. Since $uRv$, by Lemma \ref{the-condition-equivalent-to-(f7)}, $v=L(q_1)$ for some standard sequence $q_1$ of $E$ with $\prescript{\wedge}{}{p_1}\equiv \prescript{\wedge}{}{q_1}$. Write $\prescript{\wedge}{}{p_1}=p_3 p_2$, where $p_3$ is a standard sequence of $E$. Moreover, write $\prescript{\wedge}{}{q_1}=q_3 q_2$ such that $q_2 \equiv p_2$ and $q_3 \equiv p_3$. Then $p_3=\prescript{\wedge}{}{(p_2 p_1)}$ and $q_3=\prescript{\wedge}{}{(q_2 q_1)}$. Since $wu=L(p_2 p_1)$ and $wv=L(q_2 q_1)$ with $\prescript{\wedge}{}{(p_2 p_1)}=p_3\equiv q_3=\prescript{\wedge}{}{(q_2 q_1)}$, $(wu)R(wv)$ and $wu-wv$ is a relation of $I_E$ of type $(fR1)$.
\end{proof}

For a set $\mathscr{A}$ of paths of $Q_E$ and for every two vertices $x$, $y$ of $Q_E$, denote $\prescript{}{y}{\mathscr{A}}_x$ be the subset of $\mathscr{A}$ consists of paths with source $x$ and terminal $y$. For a subset $S$ of any $k$-vector space $V$, we denote by $kS$ the $k$-subspace of $V$ generated by $S$.

\begin{Lem}\label{a basis of ideal I E}
Let $E=(E,P,L,d)$ be an $f_s$-BC. For every two vertices $x$, $y$ of $Q_E$, $I_E(x,y)=k\prescript{}{y}{(\mathscr{B}_1)}_x\bigoplus k\prescript{}{y}{(\mathscr{B}_2)}_x\bigoplus k\{u-v\mid u,v\in\prescript{}{y}{\mathscr{E}}_x$ and $uRv\}$. In particular, the relations of types $(fR1), (fR2), (fR3)$ generate $I_E$ as a $k$-vector space in this case.
\end{Lem}

\begin{proof}
Each element $\eta$ of $I_E(x,y)$ is of the form $\sum\lambda_i u_i r_i v_i$, where $\lambda_i\in k^{*}$, $u_i's$, $v_i's$ are paths of $Q_E$, and $r_i's$ are relations of $I_E$ of type $(fR1)$, type $(fR2)$, or of type $(fR3)$. If $r_i$ is a relation of $I_E$ of type $(fR2)$ or of type $(fR3)$, it is straightforward to show that $u_i r_i v_i$ is also a relation of $I_E$ of type $(fR2)$ or of type $(fR3)$. Therefore $u_i r_i v_i\in \prescript{}{y}{(\mathscr{B}_1)}_x\sqcup\prescript{}{y}{(\mathscr{B}_2)}_x$. If $r_i$ is a relation of $I_E$ of type $(fR1)$, by Lemma \ref{key lemma}, either $u_i r_i$ is a relation of $I_E$ of type $(fR1)$ or $u_i r_i\in k\prescript{}{y}{(\mathscr{B}_1)}_z\bigoplus k\prescript{}{y}{(\mathscr{B}_2)}_z$, where $z$ is the source of $r_i$. Therefore either $u_i r_i v_i$ is a relation of $I_E$ of type $(fR1)$ or $u_i r_i v_i\in k\prescript{}{y}{(\mathscr{B}_1)}_x\bigoplus k\prescript{}{y}{(\mathscr{B}_2)}_x$, which also follows by Lemma \ref{key lemma}. Then we imply $u_i r_i v_i\in k\prescript{}{y}{(\mathscr{B}_1)}_x\bigoplus k\prescript{}{y}{(\mathscr{B}_2)}_x\bigoplus k\{u-v\mid u,v\in\mathscr{E}$ and $uRv\}$. So $I_E(x,y)\subseteq k\prescript{}{y}{(\mathscr{B}_1)}_x\bigoplus k\prescript{}{y}{(\mathscr{B}_2)}_x\bigoplus k\{u-v\mid u,v\in\prescript{}{y}{\mathscr{E}}_x$ and $uRv\}$. The fact that $k\prescript{}{y}{(\mathscr{B}_1)}_x\bigoplus k\prescript{}{y}{(\mathscr{B}_2)}_x\bigoplus k\{u-v\mid u,v\in\prescript{}{y}{\mathscr{E}}_x$ and $uRv\}$ is contained in $I_E(x,y)$ follows from the definition of $I_E$.
\end{proof}

\begin{Rem1}\label{full-squences-are-nonzero}
Lemma \ref{a basis of ideal I E} shows that if $E$ is an $f_s$-BC, then every path in $\mathscr{E}$ is not contained in $I_E$, and therefore it is nonzero in $\Lambda_E$. This shows that Condition $(f7)$ ensures that the nonzero paths precisely correspond to standard sequences of $E$.
\end{Rem1}

\begin{Prop}\label{a basis of J}
Let $E=(E,P,L,d)$ be an $f_s$-BC, $\mathscr{C}$ be a set of representatives of paths in $\mathscr{E}$ under equivalence relation $R$. Then for every two vertices $x$, $y$ of $Q_E$, the image of $\prescript{}{y}{\mathscr{C}}_x$ in $\Lambda(x,y)$ forms a $k$-basis of $\Lambda(x,y)$.
\end{Prop}

\begin{proof}
It follows from Lemma \ref{a basis of ideal I E} and the fact that $\prescript{}{y}{(\mathscr{B}_1)}_x\sqcup\prescript{}{y}{(\mathscr{B}_2)}_x\sqcup\prescript{}{y}{\mathscr{E}}_x$ forms a basis of $kQ_E(x,y)$.
\end{proof}

Note that for $u$, $v\in\mathscr{E}$ with $uRv$, we have $u=v$ in $\Lambda_E$, therefore the image of $\prescript{}{y}{\mathscr{C}}_x$ in $\Lambda_E(x,y)$ is independent to the choice of representatives of paths in $\mathscr{E}$ under equivalence relation $R$.

Let $E=(E,P,L,d)$ be an $f_s$-BC. For every two objects $x$, $y$ of $\Lambda_E$ and for $u\in\prescript{}{y}{\mathscr{C}}_x$, let $u=L(p)$ for some standard sequence $p$. Define $\prescript{\wedge}{}{u}$ to be the unique path in $\prescript{}{\sigma(x)}{\mathscr{C}}_y$ such that $\prescript{\wedge}{}{u}RL(\prescript{\wedge}{}{p})$, and $u^{\wedge}$ to be the unique path in $\prescript{}{x}{\mathscr{C}}_{\sigma^{-1}(y)}$ such that $u^{\wedge}RL(p^{\wedge})$.

Note that $\prescript{\wedge}{}{(-)}$ is a well-defined map from $\prescript{}{y}{\mathscr{C}}_x$ to $\prescript{}{\sigma(x)}{\mathscr{C}}_y$: if $u=L(p)=L(q)$ for standard sequences $p$, $q$ of $E$, then $p\equiv q$ and $\prescript{\wedge\wedge}{}{p}\equiv\prescript{\wedge\wedge}{}{q}$, therefore $L(\prescript{\wedge}{}{p}) R L(\prescript{\wedge}{}{q})$. Similarly, $(-)^{\wedge}$ is a well-defined map from $\prescript{}{y}{\mathscr{C}}_x$ to $\prescript{}{x}{\mathscr{C}}_{\sigma^{-1}(y)}$.

\begin{Lem}\label{inverse}
Let $E=(E,P,L,d)$ be an $f_s$-BC. For every two objects $x$, $y$ of $\Lambda_E$, $(-)^{\wedge}:\prescript{}{\sigma(x)}{\mathscr{C}}_y\rightarrow\prescript{}{y}{\mathscr{C}}_x$ is the inverse of $\prescript{\wedge}{}{(-)}:\prescript{}{y}{\mathscr{C}}_x\rightarrow\prescript{}{\sigma(x)}{\mathscr{C}}_y$.
\end{Lem}

\begin{proof}
For $u\in\prescript{}{y}{\mathscr{C}}_x$ with $u=L(p)$ for some standard sequence $p$, $\prescript{\wedge}{}{u}$ is a path in $\mathscr{C}$ with $\prescript{\wedge}{}{u}RL(\prescript{\wedge}{}{p})$. Then there exists standard sequences $q_1$, $q_2$ of $E$ with $\prescript{\wedge}{}{u}=L(q_1)$, $L(\prescript{\wedge}{}{p})=L(q_2)$ and $\prescript{\wedge}{}{q_1}\equiv\prescript{\wedge}{}{q_2}$. By Remark \ref{identical} (3), $q_{1}^{\wedge}\equiv q_{2}^{\wedge}$, and therefore $L(q_{1}^{\wedge})=L(q_{2}^{\wedge})$. Since $L(\prescript{\wedge}{}{p})=L(q_2)$, $L(p)RL(q_{2}^{\wedge})$. By the definition of $(-)^{\wedge}$, we have $(\prescript{\wedge}{}{u})^{\wedge}RL(q_{1}^{\wedge})$. So $(\prescript{\wedge}{}{u})^{\wedge}RL(p)$, where $L(p)=u$. Since both $u$ and $(\prescript{\wedge}{}{u})^{\wedge}$ belong to $\prescript{}{y}{\mathscr{C}}_x$, $(\prescript{\wedge}{}{u})^{\wedge}=u$. It can be shown similarly that $\prescript{\wedge}{}{(u^{\wedge})}=u$.
\end{proof}

For each object $x$ of $\Lambda_E$, let $w$ be the unique path in $\prescript{}{\sigma(x)}{\mathscr{C}}_x$ which corresponds to a full sequence. Define a linear form $\epsilon:\Lambda_E(x,\sigma(x))\rightarrow k$ by sending $w$ to $1$ and sending other $u\in\prescript{}{\sigma(x)}{\mathscr{C}}_x$ to zero. For every objects $x$, $y$ of $\Lambda_E$, let $\langle-,-\rangle:\Lambda_{E}(y,\sigma(x))\times\Lambda_{E}(x,y)\rightarrow k$ be the bilinear form defined by the composition $\Lambda_{E}(y,\sigma(x))\times\Lambda_{E}(x,y)\xrightarrow[]{multi}\Lambda_{E}(x,\sigma(x))\xrightarrow[]{\epsilon} k$.

\begin{Lem}\label{non-degenerate}
For every two objects $x$, $y$ of $\Lambda_E$, the bilinear form $\langle-,-\rangle:\Lambda_{E}(y,\sigma(x))\times\Lambda_{E}(x,y)\rightarrow k$ defined by $\langle a,b\rangle=\epsilon(ab)$ is non-degenerate.
\end{Lem}

\begin{proof}
It suffices to show for every $u\in\prescript{}{y}{\mathscr{C}}_x$ and $v\in\prescript{}{\sigma(x)}{\mathscr{C}}_y$,
\begin{equation*}\langle v,u\rangle= \begin{cases}
1, & v=\prescript{\wedge}{}{u}; \\
0, & \text{otherwise}.
\end{cases}
\end{equation*}
Note that for $u\in\prescript{}{y}{\mathscr{C}}_x$ and $v\in\prescript{}{\sigma(x)}{\mathscr{C}}_y$, either $vu=0$ in $\Lambda_E$ or $(vu)Rw$ for some $w\in\prescript{}{\sigma(x)}{\mathscr{C}}_x$, so the value of $\langle v,u\rangle$ can only take $0$ or $1$. When $v=\prescript{\wedge}{}{u}$, suppose that $u=L(p)$ for some standard sequence $p$. Then $\prescript{\wedge}{}{u}=L(\prescript{\wedge}{}{p})$ in $\Lambda_E$ and $\langle v,u\rangle=\epsilon(\prescript{\wedge}{}{u}u)=\epsilon(L(\prescript{\wedge}{}{p})L(p))=\epsilon(L(\prescript{\wedge}{}{p}p))=1$. Conversely, suppose that $\langle v,u\rangle=1$. Since $\epsilon(vu)=\langle v,u\rangle=1$, $vu\neq 0$ in $\Lambda_E$. By Lemma \ref{a basis of ideal I E}, $vu\in\mathscr{E}$. By the definition of $\epsilon$, $vu=L(q)$ for some full sequence $q$ of $E$. Write $q=q_2 q_1$ such that $L(q_1)=u$ and $L(q_2)=v$. Then $\prescript{\wedge}{}{u}RL(\prescript{\wedge}{}{q_1})$ and $L(\prescript{\wedge}{}{q_1})=L(q_2)=v$ imply $\prescript{\wedge}{}{u}Rv$. Since both $\prescript{\wedge}{}{u}$ and $v$ belong to $\mathscr{C}$, we have $\prescript{\wedge}{}{u}=v$.
\end{proof}

We now show that the associated locally bounded category of an $f_s$-BC is a Frobenius category.

\begin{Thm}\label{isomorphism of proj and inj}
Let $E$ be an $f_s$-BC. Then $\Lambda_{E}(-,\sigma(x))\cong D\Lambda_{E}(x,-)$ for all object $x$ of $\Lambda_E$. Therefore the associated f-BCC $\Lambda_E$ is a locally bounded Frobenius category.
\end{Thm}

\begin{proof}
For each two objects $x$, $y$ of $\Lambda_E$, by Lemma \ref{non-degenerate}, the bilinear form $\langle-,-\rangle:\Lambda_{E}(y,\sigma(x))\times\Lambda_{E}(x,y)\rightarrow k$ is non-degenerate. Therefore it induces an isomorphism $\Lambda_{E}(y,\sigma(x))\cong D\Lambda_{E}(x,y)$. Moreover, by the definition of the bilinear form, this isomorphism is natural at $y$, so we have $\Lambda_{E}(-,\sigma(x))\cong D\Lambda_{E}(x,-)$. Replace $x$ by $\sigma^{-1}(x)$, we have $\Lambda_{E}(-,x)\cong D\Lambda_{E}(\sigma^{-1}(x),-)$. Therefore $\Lambda_E$ is a locally bounded Frobenius category.
\end{proof}

\section{Fractional Brauer configuration algebra}

\begin{Def}\label{f-BCA}
Let $E$ be an $f$-BC, and let $\Lambda_{E}$ be the corresponding f-BCC. Set $$A_{E}=(\bigoplus_{x,y\in (Q_E)_0}\Lambda_{E}(x,y))^{op}.$$ We call $A_E$ the fractional Brauer configuration algebra (abbr. f-BCA) of $E$. If moreover $E$ is an $f_s$-BC (resp. $f_s$-BG, $f_{ms}$-BC, $f_{ms}$-BG), then we call $A_E$ an $f_s$-BCA (resp. $f_s$-BGA, $f_{ms}$-BCA, $f_{ms}$-BGA).
\end{Def}

By definition, $A_E$ is a locally bounded algebra (that is, there is a complete set of pairwise orthogonal primitive idempotents $\{1_x\mid x\in (Q_E)_0\}$ such that $(A_E)1_x$ and $1_x(A_E)$ are finite-dimensional over $k$ for all $x\in (Q_E)_0$) since the category $\Lambda_E$ is locally bounded and $A_E$ is isomorphic to $(kQ_E/I_E)^{op}\cong kQ_{E}^{op}/I_{E}^{op}$. If moreover $E$ is a finite set, then $A_E$ is a finite-dimensional algebra, and the category of finitely generated left $A_E$-modules is equivalent to the category of finitely generated $\Lambda_E$-modules.

When $E$ is a finite $f_s$-BC, we may extend the linear forms $\epsilon:\Lambda_{E}(x,\sigma(x))\rightarrow k$ (the definition of $\epsilon$ is given before Lemma \ref{non-degenerate}) to a linear form of $A_E$, which is also denoted by $\epsilon$. Let $\langle-,-\rangle:A_E\times A_E\rightarrow k$ be the bilinear form defined by the composition $A_E\times A_E\xrightarrow[]{multi}A_E\xrightarrow[]{\epsilon} k$. For each $x$, $y\in (Q_E)_0$, the restriction of this bilinear form to $\Lambda_{E}(x,y)\times\Lambda_{E}(y,\sigma(x))$ is non-degenerate by Lemma \ref{non-degenerate} and the restriction of $\langle-,-\rangle:A_E\times A_E\rightarrow k$ to $\Lambda_{E}(x,y)\times\Lambda_{E}(y',x')$ is zero, whenever $y\neq y'$ or $x'\neq \sigma(x)$. Therefore the bilinear form $\langle-,-\rangle:A_E\times A_E\rightarrow k$ is non-degenerate with the property that $\langle a\cdot b,c\rangle=\langle a,b\cdot c\rangle$ for all $a,b,c\in A_E$, where $a\cdot b$ denotes the product of $a$ and $b$ in $A_E$. Then we have:

\begin{Prop}\label{fsBCA-is-Frobenius}
If $E$ is a finite $f_s$-BC, then $A_{E}$ is a finite-dimensional Frobenius algebra with the linear form $\epsilon:A_E\rightarrow k$ defined as above.
\end{Prop}

\medskip
In Example \ref{example0}, $E$ is an $f_{ms}$-BG and $A_E$ is a finite-dimensional special biserial Frobenius algebra. The structures of indecomposable projective modules of $A_{E}$ are

$$P_1= \xymatrix@R=0.5pc@C=0.8pc  {
	&1\ar@{-}[dl]\ar@{-}[dr]& \\
    3\ar@{-}[dr]&&3\ar@{-}[dl] \\
    &2&
}, \quad \quad
P_2= \xymatrix@R=0.5pc@C=0.8pc  {
	&2\ar@{-}[dl]\ar@{-}[dr]& \\
    1\ar@{-}[dr]&&1\ar@{-}[dl] \\
    &3&
}, \quad \quad
P_3= \xymatrix@R=0.5pc@C=0.8pc  {
	&3\ar@{-}[dl]\ar@{-}[dr]& \\
    2\ar@{-}[dr]&&2\ar@{-}[dl] \\
    &1&
}.$$

In Example \ref{example1}, $E$ is an $f_s$-BG and $A_E$ is a finite-dimensional symmetric algebra but is not multiserial. The structures of indecomposable projective modules of $A_{E}$ are

$$P_1= \xymatrix@R=0.5pc@C=0.8pc  {
	&1\ar@{-}[dl]\ar@{-}[dr]& \\
    3\ar@{-}[dr]&&4\ar@{-}[dl] \\
    &2\ar@{-}[d]& \\
    &1&
}, \quad\quad
P_2= \xymatrix@R=0.5pc@C=0.8pc  {
	&2\ar@{-}[d]& \\
    &1\ar@{-}[dl]\ar@{-}[dr]& \\
    3\ar@{-}[dr]&&4\ar@{-}[dl] \\
    &2&
}, \quad\quad
P_3= \xymatrix@R=0.5pc@C=0.8pc  {
	3\ar@{-}[d] \\
    2\ar@{-}[d] \\
    1\ar@{-}[d] \\
    3
}, \quad\quad
P_4= \xymatrix@R=0.5pc@C=0.8pc  {
	4\ar@{-}[d] \\
    2\ar@{-}[d] \\
    1\ar@{-}[d] \\
    4
}.$$

In Example \ref{example2}, $E$ is an f-BC but not an $f_s$-BC (see remarks after Remark \ref{two-conditions-are equivalent}), and $A_E\cong k$ (see remarks after Remark \ref{interpretations-to-the-ideal-$I_E$}).

In Example \ref{example3}, $E$ is an f-BC but not an $f_s$-BC (see remarks after Remark \ref{two-conditions-are equivalent}). The structures of indecomposable projective modules of $A_{E}$ are

$$P_1= \xymatrix@R=0.5pc@C=0.8pc  {
	&1\ar@{-}[dl]\ar@{-}[dr]& \\
    5\ar@{-}[dr]&&6\ar@{-}[dl] \\
    &4\ar@{-}[d]& \\
    &2&
},
P_2= \xymatrix@R=0.5pc@C=0.8pc  {
	&2\ar@{-}[d]& \\
    &1\ar@{-}[dr]\ar@{-}[dl]& \\
    5\ar@{-}[dr]&&6\ar@{-}[dl] \\
    &4\ar@{-}[d]& \\
    &2&
},
P_3= \xymatrix@R=0.5pc@C=0.8pc  {
	3\ar@{-}[d] \\
    1\ar@{-}[d] \\
    5
},
P_4= \xymatrix@R=0.5pc@C=0.8pc  {
	&4\ar@{-}[dl]\ar@{-}[dr]& \\
    2\ar@{-}[dr]&&3\ar@{-}[dl] \\
    &1\ar@{-}[d]& \\
    &5&
},
P_5= \xymatrix@R=0.5pc@C=0.8pc {
	&5\ar@{-}[d]& \\
    &4\ar@{-}[dr]\ar@{-}[dl]& \\
    2\ar@{-}[dr]&&3\ar@{-}[dl] \\
    &1\ar@{-}[d]& \\
    &5&
},
P_6= \xymatrix@R=0.5pc@C=0.8pc {
	6\ar@{-}[d] \\
    4\ar@{-}[d] \\
    2
}.$$
Therefore $A_{E}$ is finite-dimensional but not self-injective.

In Example \ref{example4}, $E$ is an $f_{ms}$-BC and $A_E$ is a locally bounded special biserial Frobenius algebra. The structures of indecomposable projective modules of $A_{E}$ are

$$P_{3i}= \xymatrix@R=1pc@C=0pc  {
	&3i\ar@{-}[dl]\ar@{-}[dr]& \\
    3i-1\ar@{-}[dr]&&3i-2\ar@{-}[dl] \\
    &3i-3&
}, \quad\quad
P_{3i+1}= \xymatrix@R=1pc@C=0pc  {
	3i+1\ar@{-}[d] \\
    3i\ar@{-}[d] \\
    3i-2
}, \quad\quad
P_{3i+2}= \xymatrix@R=1pc@C=0pc  {
	3i+2\ar@{-}[d] \\
    3i\ar@{-}[d] \\
    3i-1
},
$$
where $i\in\mathbb{Z}$.

\medskip
Let $E$ be a finite $f_s$-BC. The automorphism $\sigma$ of $\Lambda_E$ induces an automorphism of $A_E$, which is also denoted by $\sigma$.
Note that the automorphism $\sigma$ preserves the bilinear form $\langle-,-\rangle$ on $A_E$, that is, $\langle\sigma(a),\sigma(b)\rangle=\langle a,b\rangle$ for all $a,b\in A_E$.

\begin{Lem}\label{epsilon}
If $E$ is a finite $f_s$-BC, then $\epsilon(b\cdot \sigma(a))=\epsilon(a\cdot b)$ for all $a,b\in A_E$.
\end{Lem}

\begin{proof}
Let $a=\sum_{x,y\in (Q_E)_0}\prescript{}{y}{a}_x$ and $b=\sum_{x,y\in (Q_E)_0}\prescript{}{y}{b}_x$, where $\prescript{}{y}{a}_x$, $\prescript{}{y}{b}_x \in\Lambda_{E}(x,y)$. Then $\epsilon(a\cdot b)=\sum_{x,y\in (Q_E)_0}\epsilon(\prescript{}{y}{a}_x\cdot\prescript{}{\sigma(x)}{b}_y)$ and
$$\epsilon(b\cdot\sigma(a))=\sum_{x,y\in (Q_E)_0}\epsilon(\prescript{}{\sigma(x)}{b}_y\cdot\prescript{}{\sigma(y)}{\sigma(a)}_{\sigma(x)})=\sum_{x,y\in (Q_E)_0}\epsilon(\prescript{}{\sigma(x)}{b}_y\cdot\sigma(\prescript{}{y}{a}_{x})).$$
To show $\epsilon(b\cdot\sigma(a))=\epsilon(a\cdot b)$, it suffices to show that $\epsilon(\prescript{}{y}{a}_{x}\cdot\prescript{}{\sigma(x)}{b}_{y})
=\epsilon(\prescript{}{\sigma(x)}{b}_{y}\cdot\sigma(\prescript{}{y}{a}_{x}))$ for each $x$, $y\in (Q_E)_0$. We may assume that $\prescript{}{y}{a}_{x}\in\prescript{}{y}{\mathscr{C}}_{x}$ and $\prescript{}{\sigma(x)}{b}_{y}\in\prescript{}{\sigma(x)}{\mathscr{C}}_{y}$.  For $u\in\prescript{}{y}{\mathscr{C}}_{x}$ and $v\in\prescript{}{\sigma(x)}{\mathscr{C}}_{y}$, by the proof of Lemma \ref{non-degenerate},
\begin{equation*}
\epsilon(u\cdot v)=\epsilon(vu)= \begin{cases}
1, & v=\prescript{\wedge}{}{u}; \\
0, & \text{otherwise}.
\end{cases}
\end{equation*}

Suppose that $u=L(p)$ for some standard sequence $p$, then $\prescript{\wedge}{}{u}RL(\prescript{\wedge}{}{p})$. By Lemma \ref{the-condition-equivalent-to-(f7)}, $\prescript{\wedge}{}{u}=L(r)$ for some standard sequence $r$ with $\prescript{\wedge}{}{r}\equiv\prescript{\wedge\wedge}{}{p}$. Then $\prescript{\wedge\wedge}{}{u}RL(\prescript{\wedge}{}{r})$, where $L(\prescript{\wedge}{}{r})=L(\prescript{\wedge\wedge}{}{p})=\sigma(u)$. Therefore $\sigma(u)=\prescript{\wedge\wedge}{}{u}$ in $\Lambda_E$. If $v=\prescript{\wedge}{}{u}$, then $\epsilon(v\cdot\sigma(u))=\epsilon(\prescript{\wedge\wedge}{}{u}\prescript{\wedge}{}{u})=1=\epsilon(u\cdot v)$. If $v\neq \prescript{\wedge}{}{u}$, by Lemma \ref{inverse}, $\prescript{\wedge}{}{v}\neq\prescript{\wedge\wedge}{}{u}$. Then $\epsilon(v\cdot\sigma(u))=\epsilon(\prescript{\wedge\wedge}{}{u}v)=0=\epsilon(u\cdot v)$. Therefore $\epsilon(\prescript{}{y}{a}_{x}\cdot\prescript{}{\sigma(x)}{b}_{y})
=\epsilon(\prescript{}{\sigma(x)}{b}_{y}\cdot\sigma(\prescript{}{y}{a}_{x}))$.
\end{proof}

\begin{Prop}
If $E$ is a finite $f_s$-BC, then the automorphism $\sigma$ of $A_E$ is equal to the usual Nakayama automorphism of the self-injective algebra $A_E$.
\end{Prop}

\begin{proof}
Let $\nu_{A}$ be the Nakayama automorphism of $A_E$. Then $\epsilon(a\cdot b)=\epsilon(b\cdot\nu_{A}(a))$ for all $a$, $b\in A_E$. By Lemma \ref{epsilon}, $\epsilon(b\cdot \sigma(a))=\epsilon(a\cdot b)$ for all $a$, $b\in A_E$. Since $\langle a,b \rangle=\epsilon(a\cdot b)$ and $\langle-,-\rangle:A_E\times A_E\rightarrow k$ is non-degenerate, $\sigma=\nu_{A}$.
\end{proof}

\begin{Prop}\label{symmetric-algebra}
If $E$ is a finite $f_s$-BC with integral f-degree (for example, if $E$ is a finite $f_{ms}$-BC with integral f-degree or if $E$ is a finite $f_{s}$-BG with integral f-degree), then $A_E$ is a finite-dimensional symmetric algebra.
\end{Prop}

\begin{proof}
Since $E$ has integral f-degree, the automorphism $\sigma$ of $A_E$ is identity. By Lemma \ref{epsilon}, the linear form $\epsilon$ of $A_E$ is symmetric.
\end{proof}

\begin{Rem1} \label{counter-examples}
\begin{itemize}
  \item[(1)] There exist self-injective algebras which are not Morita equivalent to $f_s$-BCA. For example, if $k$ is an algebraically closed field, then every basic indecomposable nonstandard representation-finite self-injective algebra is not isomorphic to $f_s$-BCA (see Theorem \ref{RFS-algebra=r.f.fsBCA}). Another example which is not Morita equivalent to $f_s$-BCA is the algebra $k\langle x,y\rangle/\langle x^2,y^2,xy-\lambda yx\rangle$, where $\lambda\neq 1$.
  \item[(2)] By Corollary \ref{$f_{ms}$-BCA-of-integral-f-degree-is-equal-to-BCA}, an $f_{ms}$-BCA with integral f-degree is equal to a BCA at least over an algebraically closed field. However, the $f_{s}$-BGAs with integral f-degree beyond the scope of BCAs, they are also symmetric but not special biserial in general and have been studied in \cite{X}.
\end{itemize}

\end{Rem1}

\section{The Gabriel quiver and admissible relations of a fractional Brauer configuration category in type S}

Let $\Lambda$ be a locally bounded $k$-category. According to \cite[Section 2.1]{BG}, $\Lambda$ is isomorphic to the form $kQ/I$, where $Q$ is a locally finite quiver and $I$ is an admissible ideal of path category $kQ$, that is, for each $x\in Q_0$, there exists a positive integer $N_x$ with $kQ^{\geq N_x}(x,-)\subseteq I(x,-)\subseteq kQ^{\geq 2}(x,-)$ and $kQ^{\geq N_x}(-,x)\subseteq I(-,x)\subseteq kQ^{\geq 2}(-,x)$, where $kQ^{\geq n}$ denotes the ideal of $kQ$ generated by paths of length $\geq n$. Moreover, such quiver $Q$ is uniquely determined by $\Lambda$, which is called the Gabriel quiver of $\Lambda$.

By \cite[Section 2.1]{BG}, the Gabriel quiver $Q$ and the admissible ideal $I$ of $\Lambda$ can be constructed as follows. The vertices of $Q$ are the objects of $\Lambda$. Denote $\mathrm{rad}_{\Lambda}$ the radical of the $k$-category $\Lambda$. For every objects $x,y$ of $\Lambda$, choose morphisms $f_1,\cdots,f_m$ in $\mathrm{rad}_{\Lambda}(x,y)$ such that the images of $f_1,\cdots,f_m$ in $\mathrm{rad}_{\Lambda}(x,y)/\mathrm{rad}^{2}_{\Lambda}(x,y)$ form a $k$-basis of it. Then there are $m$ arrows $\alpha_1,\cdots,\alpha_m$ from $x$ to $y$ in $Q$. Moreover, the functor $\rho:kQ\rightarrow\Lambda$ sending each vertex of $Q$ to the associated object in $\Lambda$ and sending each arrow $\alpha_i$ to the morphism $f_i$ is full, which induces an isomorphism $kQ/I\rightarrow\Lambda$, where $I=\ker\rho$ is the admissible ideal.

In this section, we will determine the Gabriel quiver and admissible relations of $\Lambda_E=kQ_E/I_E$, where $E$ is an $f_s$-BC. As the following lemma shows that, the Gabriel quiver and the admissible ideal are obtained from a reduction procedure from the quiver $Q_E$ and the ideal $I_E$.

Let $E=(E,P,L,d)$ be an $f_s$-BC. Let $\mathscr{E}$ be the set of paths of the quiver $Q_E$ with an equivalence relation $R$ defined in Definition \ref{relation R}. Note that each arrow $\alpha$ of $Q_E$ from $x$ to $y$ belongs to $\prescript{}{y}{\mathscr{E}}_x$.

\begin{Def} \label{reduced-arrow}
\begin{itemize}
  \item[(1)] We call an arrow $\alpha$ of $Q_E$ from $x$ to $y$ reduced, if there exists a path $p\in\prescript{}{y}{\mathscr{E}}_x$ of length $\geq 2$ such that $\alpha R p$.
  \item[(2)] Denote $\prescript{}{y}{\mathscr{N}}_x$ a complete set of representatives of non-reduced arrows of $Q_E$ from $x$ to $y$ under the equivalence relation $R$, and define a subquiver $Q'_E$ of $Q_E$: $$(Q'_E)_0=(Q_E)_0, \quad (Q'_E)_1=\bigsqcup_{x,y\in (Q_E)_0}\prescript{}{y}{\mathscr{N}}_x.$$
  \item[(3)] Denote $\rho:kQ'_E\rightarrow \Lambda_{E}=kQ_E/I_E$ the natural $k$-linear functor, and let $I'_E:=\mathrm{ker}\rho$.
      \end{itemize}
\end{Def}

Note that by the property (5) in Definition \ref{BC}, if $E$ is a BC (viewed as an f-BC) and $e$ is an angle such that the corresponding vertex $G\cdot e$ is truncated, then the arrow $L(e)$ of $Q_E$ is reduced.

\begin{Lem}\label{Gabriel quiver}
$Q'_E$ is the Gabriel quiver of $\Lambda_E$, the functor $\rho: kQ'_E\rightarrow \Lambda_{E}$ is dense and full, and the kernel $I'_E$ of $\rho $ is an admissible ideal of $kQ'_E$. In particular, the $f_s$-BC category $\Lambda_{E}$ is isomorphic to the category $kQ'_E/I'_E$.
\end{Lem}

\begin{proof}
The fact that $\rho$ is dense is clear since $Q'_E$ and $Q_E$ have the same vertices. To show that $\rho$ is full, it suffices to show that for each $x,y\in (Q_E)_0$, the image of $\prescript{}{y}{\mathscr{N}}_x$ in $J(x,y)/J^2(x,y)$ forms a $k$-basis of $J(x,y)/J^2(x,y)$, where $J$ is the ideal of $\Lambda_{E}$ generated by the arrows of $Q_E$ ($J$ is also the radical of $\Lambda_{E}$ by Theorem \ref{f-BC algebra is locally bounded}).

By Proposition \ref{a basis of J}, the set of paths $$
\prescript{}{y}{\mathscr{D}}_x=\begin{cases}
\prescript{}{y}{\mathscr{C}}_x-\{1_x\}, \text{ if } x=y; \\
\prescript{}{y}{\mathscr{C}}_x, \text{ otherwise}
\end{cases}
$$
forms a $k$-basis of $J(x,y)$, where $\mathscr{C}$ is a set of representatives of paths in $\mathscr{E}$ under equivalence relation $R$. We may assume that $\prescript{}{y}{\mathscr{N}}_x$ is contained in $\prescript{}{y}{\mathscr{D}}_x$. For each $p\in\prescript{}{y}{\mathscr{D}}_x-\prescript{}{y}{\mathscr{N}}_x$, either $p$ is a reduced arrow of $Q_E$ or $p$ is a path of length $\geq 2$, therefore $p\in J^2(x,y)$. So the image of $\prescript{}{y}{\mathscr{N}}_x$ in $J(x,y)/J^2(x,y)$ generates the whole space. Suppose that $\sum_{i=1}^{n}\lambda_i\alpha_i\in J^2(x,y)$, where $\lambda_i\in k$ and $\alpha_i\in\prescript{}{y}{\mathscr{N}}_x$ for each $1\leq i\leq n$, then we may assume that $\sum_{i=1}^{n}\lambda_i\alpha_i=\sum_{j=1}^{m}\mu_j p_j$ in $\Lambda_E$, where $\mu_j\in k$ and $p_j$ is a path of $Q_E$ of length $\geq 2$ for each $1\leq j\leq m$. According to Lemma \ref{a basis of ideal I E}, we may assume that each $p_j$ belongs to $\prescript{}{y}{\mathscr{E}}_x$. For each $1\leq j\leq m$, let $q_j$ be the path of $\prescript{}{y}{\mathscr{D}}_x$ such that $p_j R q_j$. Since $l(p_j)\geq 2$, $q_j\notin\prescript{}{y}{\mathscr{N}}_x$. Since $\prescript{}{y}{\mathscr{D}}_x$ forms a $k$-basis of $J(x,y)$, we have $\sum_{i=1}^{n}\lambda_i\alpha_i=\sum_{j=1}^{m}\mu_j q_j=0$ in $\Lambda_E$. Then $\lambda_i=0$ for each $1\leq i\leq n$ and the image of $\prescript{}{y}{\mathscr{N}}_x$ in $J(x,y)/J^2(x,y)$ is linearly independent.
\end{proof}

For each set $\mathscr{A}$ of paths of $Q_E$, denote $\mathscr{A}'$ the subset of $\mathscr{A}$ formed by the paths in $\mathscr{A}$ which are also paths of $Q'_E$. Then for every vertices $x$, $y$ of $Q'_E$, $\prescript{}{y}{(\mathscr{B}'_1)}_x\sqcup\prescript{}{y}{(\mathscr{B}'_2)}_x\sqcup\prescript{}{y}{\mathscr{E}'}_x$ forms a $k$-basis of $kQ'_E(x,y)$ (cf. Remark \ref{disjoint-union}). Similar to Lemma \ref{a basis of ideal I E}, we have

\begin{Lem}\label{a basis of ideal I' E}
Let $E=(E,P,L,d)$ be an $f_s$-BC. For every vertices $x$, $y$ of $Q'_E$, $I'_E(x,y)=k\prescript{}{y}{(\mathscr{B}'_1)}_x\bigoplus k\prescript{}{y}{(\mathscr{B}'_2)}_x\bigoplus k\{u-v\mid u,v\in\prescript{}{y}{\mathscr{E}'}_x$ and $uRv\}$.
\end{Lem}

\begin{proof}
By Lemma \ref{a basis of ideal I E} we have $I_E(x,y)=k\prescript{}{y}{(\mathscr{B}_1)}_x\bigoplus k\prescript{}{y}{(\mathscr{B}_2)}_x\bigoplus k\{u-v\mid u,v\in\prescript{}{y}{\mathscr{E}}_x$ and $uRv\}$. Since $I'_E$ is the kernel of the $k$-linear functor $\rho:kQ'_E\rightarrow \Lambda_{E}=kQ_E/I_E$, the subspace $k\prescript{}{y}{(\mathscr{B}'_1)}_x\bigoplus k\prescript{}{y}{(\mathscr{B}'_2)}_x\bigoplus k\{u-v\mid u,v\in\prescript{}{y}{\mathscr{E}'}_x$ and $uRv\}$ of $kQ'_E(x,y)$ is contained in $I'_E(x,y)$.

Conversely, for each $r\in I'_E(x,y)$, we may write $$r=\sum_{p\in\prescript{}{y}{(\mathscr{B}'_1)}_x\sqcup\prescript{}{y}{(\mathscr{B}'_2)}_x}\lambda_p p+\sum_{q\in\prescript{}{y}{\mathscr{E}'}_x}\mu_q q,$$ where $\lambda_p,\mu_q\in k$. Then $$\sum_{p\in\prescript{}{y}{(\mathscr{B}'_1)}_x\sqcup\prescript{}{y}{(\mathscr{B}'_2)}_x}\lambda_p p+\sum_{q\in\prescript{}{y}{\mathscr{E}'}_x}\mu_q q\in I_E(x,y).$$ Since each $p\in\prescript{}{y}{(\mathscr{B}'_1)}_x\sqcup\prescript{}{y}{(\mathscr{B}'_2)}_x$ belongs to $I_E(x,y)$, $\sum_{q\in\prescript{}{y}{\mathscr{E}'}_x}\mu_q q\in I_E(x,y)$. By Lemma \ref{a basis of ideal I E} we imply that $\sum_{q\in\prescript{}{y}{\mathscr{E}'}_x}\mu_q q$ belongs to the subspace $k\{u-v\mid u,v\in\prescript{}{y}{\mathscr{E}}_x$ and $uRv\}$ of $I_E(x,y)$. Therefore for each $q\in\prescript{}{y}{\mathscr{E}'}_x$, $$\sum_{l\in\prescript{}{y}{\mathscr{E}'}_x,lRq}\mu_l=0.$$ So $\sum_{q\in\prescript{}{y}{\mathscr{E}'}_x}\mu_q q\in k\{u-v\mid u,v\in\prescript{}{y}{\mathscr{E}'}_x$ and $uRv\}$ and $r$ belongs to the subspace \\ $k\prescript{}{y}{(\mathscr{B}'_1)}_x\bigoplus k\prescript{}{y}{(\mathscr{B}'_2)}_x\bigoplus k\{u-v\mid u,v\in\prescript{}{y}{\mathscr{E}'}_x$ and $uRv\}$ of $kQ'_E(x,y)$.
\end{proof}

\begin{Rem1}{\rm(Compare to Remark \ref{two-quivers-with-relations-are-identical})}\label{original-quiver-with-relations-of-BCA}
If $E$ is a Brauer configuration, then the quiver $Q'_E$ and the ideal $I'_E$ are just the quiver and the ideal given in \cite[Section 2]{GS}.
\end{Rem1}

Note that Examples \ref{example0}, \ref{example1} and \ref{example4} are of type S. In Examples \ref{example0} and \ref{example4}, $Q'_E$ is the same as $Q_{E}$ respectively, and in Example \ref{example1}, $Q'_E$ is obtained from $Q_{E}$ by removing the loops.

\begin{Def}{\rm(cf. \cite{GS2})}
We call a locally bounded category $\Lambda$ special multiserial if $\Lambda\cong kQ/I$ for some locally finite quiver $Q$ and some admissible ideal $I$, such that for each arrow $\alpha$ of $Q$, there exists at most one arrow $\beta$ (resp. $\gamma$) of $Q$ such that $\beta\alpha\notin I$ (resp. $\alpha\gamma\notin I$).
\end{Def}

\begin{Prop}\label{special multiserial}
If $E$ is an $f_{ms}$-BC, then $\Lambda_E$ is a locally bounded special multiserial Frobenius category. In particular, if $E$ is an $f_{ms}$-BG, then $\Lambda_E$ is a locally bounded special biserial Frobenius category.
\end{Prop}

\begin{proof}
By Lemma \ref{Gabriel quiver}, $\Lambda_E\cong kQ'_E/I'_E$, where $Q'_E$ is a locally finite quiver and $I'_E$ is an admissible ideal of $kQ'_E$. Let $\alpha=L(e)$ be an arrow of $Q'_E$. If $\beta\alpha\notin I'_E$ for some arrow $\beta$ of $Q'_E$, by Lemma \ref{a basis of ideal I' E} we have $\beta\alpha\notin\mathscr{B}'_1$, so $\beta\alpha=L(g\cdot h)L(h)$ for some $h\in E$. Since the partition $L$ of $E$ is trivial, $e=h$ and $\beta=L(g\cdot e)$. Similarly, if $\alpha\gamma\notin I'_E$ for some arrow $\gamma$ of $Q'_E$, then $\gamma=L(g^{-1}\cdot e)$. So there exists at most one arrow $\beta$ (resp. $\gamma$) of $Q'_E$ such that $\beta\alpha\notin I'_E$ (resp. $\alpha\gamma\notin I'_E$), and $\Lambda_E$ is special multiserial. The fact that $\Lambda_E$ is Frobenius follows from Theorem \ref{isomorphism of proj and inj}.
\end{proof}

\begin{Cor}\label{$f_{ms}$-BCA-of-integral-f-degree-is-equal-to-BCA}
Let $k$ be an algebraically closed field. If $E$ is a finite $f_{ms}$-BC with integral f-degree, then $A_E$ is a BCA. Conversely, each BCA is isomorphic to an f-BCA $A_E$, where $E$ is a finite $f_{ms}$-BC with integral f-degree.
\end{Cor}

\begin{proof} If $E$ is a finite $f_{ms}$-BC with integral f-degree, then by Propositions \ref{special multiserial} and \ref{symmetric-algebra}, $A_E$ is a finite-dimensional special multiserial symmetric algebra. Now by the result of \cite{GS2}, over an algebraically closed field, the class of symmetric special multiserial algebras coincides with the class of BCAs. Conversely, if $A$ be the BCA associated to the BC $\Gamma$, then we can consider $\Gamma$ as a finite $f_{ms}$-BC $E=(E,P,L,d)$ with integral f-degree. Let $E'$ be the $G$-set defined as follows: $E'=E$ as sets; the action of $G=\langle g\rangle$ on $E'$ is given by $g^i(e):=g^{-i}\cdot e$ for each $e\in E'$, where $g^i\cdot e$ denotes the action of $g^i$ on $e$ given by the $G$-set structure of $E$. Then $E'=(E',P,L,d)$ is also a finite $f_{ms}$-BC with integral f-degree, and $A\cong A_{E'}$.
\end{proof}

One reason for us to define f-BCAs is that BCAs are not closed under derived equivalence.

\begin{Ex1}\label{example-not-invariant-under-de}
Let $E=\{1,1',1'',2,2',3,3',4,4',5,5'\}$. Define the group action on $E$ by $g\cdot 1=2$, $g\cdot 2=3$, $g\cdot 3=1$, $g\cdot 1'=4'$, $g\cdot 4'=1'$, $g\cdot 1''=5'$, $g\cdot 5'=1''$, $g\cdot 2'=2'$, $g\cdot 3'=3'$, $g\cdot 4=4$, $g\cdot 5=5$. Define $P(1)=\{1,1',1''\}$, $P(2)=\{2,2'\}$, $P(3)=\{3,3'\}$, $P(4)=\{4,4'\}$, $P(5)=\{5,5'\}$, and $L(e)=\{e\}$ for every $e\in E$. The f-degree of $E$ is defined to be trivial. Then $E$ is a Brauer configuration, and $A_E$ is a BCA, which is given by the quiver
$$
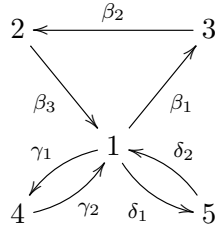
\vcenter{
\xymatrix@R=1pc@C=2pc  {
	2\ar[ddr]_{\beta_3}&&3\ar[ll]_{\beta_2} \\
    && \\
    &1\ar[uur]_{\beta_1}\ar@/_/[dl]_{\gamma_1}\ar@/_/[dr]_{\delta_1}& \\
    4\ar@/_/[ur]_{\gamma_2}&&5\ar@/_/[ul]_{\delta_2}
}} $$
with relations $0=\beta_1\gamma_2=\gamma_1\beta_3=\beta_1\delta_2=\delta_1\beta_3=\gamma_1\delta_2=\delta_1\gamma_2=\beta_2\beta_1\beta_3\beta_2$, $\beta_3\beta_2\beta_1=\gamma_2\gamma_1=\delta_2\delta_1$.

$A_E$ is a representation-finite self-injective algebra of type $(D_5,1,1)$, which is derived equivalent to the algebra $B$, where $B$ is given by the quiver
$$
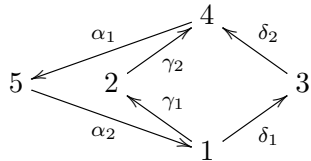
\vcenter{
\xymatrix@R=1pc@C=2pc  {
	&&4\ar[dll]_{\alpha_1}& \\
    5\ar[drr]_{\alpha_2}&2\ar[ur]_{\gamma_2}&&3\ar[ul]_{\delta_2} \\
    &&1\ar[ul]_{\gamma_1}\ar[ur]_{\delta_1}&
}} $$
with relations $\gamma_2\gamma_1=\delta_2\delta_1$, $0=\delta_1\alpha_2\alpha_1\gamma_2=\gamma_1\alpha_2\alpha_1\delta_2=
\alpha_2\alpha_1\gamma_2\gamma_1\alpha_2=\alpha_1\gamma_2\gamma_1\alpha_2\alpha_1$.

Since $B$ is not multiserial, it is not a BCA. However, $B$ is an $f_s$-BCA, whose fractional Brauer configuration $E'=(E',P',L',d')$ is given as follows:
$E'=\{1,1',2,3',4,4',5,5'\}$. The action of $\langle g\rangle$ on $E'$ is given by $g\cdot 1=5$, $g\cdot 5=4$, $g\cdot 4=2$, $g\cdot 2=1$, $g\cdot 1'=5'$, $g\cdot 5'=4'$, $g\cdot 4'=3'$, $g\cdot 3'=1'$. The partition $P'$ on $E'$ is given by $P'(1)=\{1,1'\}$, $P'(2)=\{2\}$, $P'(3')=\{3'\}$, $P'(4)=\{4,4'\}$, $P'(5)=\{5,5'\}$. The partition $L'$ on $E'$ is given by $L'(4)=\{4,4'\}$, $L'(5)=\{5,5'\}$, and $L'(e)=\{e\}$ for other $e\in E'$. The f-degree of $E$ is defined to be trivial.
\end{Ex1}

\section{Fractional Brauer configuration algebras and representation-finite self-injective algebras}

Throughout this section $k$ will be an algebraically closed field.

Let $Q$ be a locally finite connected quiver without double arrows, $I$ be an admissible ideal of path category $kQ$. The pair $(Q,I)$ is called a quiver with relations.

Recall from \cite[Definition 2.2]{BG} that a locally bounded category $\Lambda$ is called locally representation-finite if for every object $x$ of $\Lambda$, the number of isomorphism classes of finitely generated indecomposable $\Lambda$-module $l$ such that $l(x)\neq 0$ is finite. Denote the Gabriel quiver of $\Lambda$ by $Q_\Lambda$. Throughout this section we assume that the Gabriel quiver $Q_\Lambda$ is connected. Note that if $\Lambda$ is locally representation-finite and connected, then it admits the Auslander-Reiten sequences and its Auslander-Reiten quiver is also connected (see \cite{BG}).

For a translation quiver $\Gamma$, we let $k\Gamma$ be its path category, and let $k(\Gamma)$ be the mesh category of $\Gamma$, which is a factor category of $k\Gamma$ by the mesh ideal. For a locally bounded category $\Lambda$, we denote by ind$\Lambda$ the category formed by chosen representatives of the finitely generated indecomposable modules.

\begin{Def}{\rm (\cite[Definition 5.1]{BG})}\label{standard l.r.f. category}
A locally representation-finite category $\Lambda$ is said to be standard if $k(\Gamma_{\Lambda})\cong \mathrm{ind}\Lambda$, where $\Gamma_{\Lambda}$ is the Auslander-Reiten quiver of $\Lambda$.
\end{Def}

\begin{Def}{\rm (\cite[Definition 1.3]{MV-DLP})}\label{minimal relation}
Let $(Q,I)$ be a quiver with relations. A relation $\rho=\sum_{i=1}^{n}\lambda_i u_i\in I(x,y)$ with $\lambda_i\in k^{*}$ and $u_i$ a path from $x$ to $y$, is a minimal relation if $n\geq 2$ and for every non-empty proper subset $K$ of $\{1,\cdots,n\}$, $\sum_{i\in K}\lambda_i u_i\notin I(x,y)$.
\end{Def}

For a quiver with relations $(Q,I)$, \cite[Lemma 2.3]{MV-DLP} introduces two conditions: \\ $(D)$ If $\rho=\sum_{i=1}^{n}\lambda_i u_i\in I$ is a minimal relation, then for every two different $i,j\in\{1,\cdots,n\}$ there exists $c\in k^{*}$ such that $u_i+cu_j\in I$. \\ $(C)$ Let $x$, $y$, $z\in Q_0$, $u$ be a path of $Q$ with source $x$ and terminal $y$, and $v$, $w$ be two paths of $Q$ with source $y$ and terminal $z$, such that $vu$, $wu\notin I$. Then for $\lambda\in k^{*}$, $vu+\lambda wu\in I$ if and only if $v+\lambda w\in I$.

According to \cite[Corollary 3.9]{MV-DLP}, a locally representation-finite $k$-category $\Lambda$ is standard if and only if $\Lambda\cong kQ/I$ for some quiver with relations $(Q,I)$ which satisfies conditions $(D)$ and $(C)$.

\begin{Prop}\label{l.r.f. fsBCC is standard}
Let $E=(E,P,L,d)$ be an $f_s$-BC such that $\Lambda_{E}$ is locally representation-finite. Then $(Q'_E,I'_E)$ satisfies conditions $(D)$ and $(C)$, hence $\Lambda_E\cong kQ'_E/I'_E$ is standard.
\end{Prop}

\begin{proof}
By Lemma \ref{a basis of ideal I' E}, any relation $r\in I'_E(x,y)$ can be written as $\sum_{i=1}^{n}\lambda_i u_i+\sum_{j=1}^{m}\mu_j v_j+\sum_{k=1}^{d}\nu_k (p_k-q_k)$, where $\lambda_i$, $\mu_j$, $\nu_k\in k^{*}$, $u_i\in\prescript{}{y}{(\mathscr{B}'_1)}_x$, $v_j\in\prescript{}{y}{(\mathscr{B}'_2)}_x$, and $p_k$, $q_k\in\prescript{}{y}{\mathscr{E}'}_x$ with $p_k R q_k$. If $r$ is minimal, then $n$, $m=0$. Therefore $r$ is of the form $\sum_{i=1}^{N}\sum_{j=1}^{d_i}c_{ij}w_{ij}$, where $c_{ij}\in k^{*}$, $w_{ij}\in\prescript{}{y}{\mathscr{E}'}_x$, and $w_{ij}Rw_{i',j'}$ if and only if $i=i'$. Since $r\in I'_E(x,y)$, by Lemma \ref{a basis of ideal I' E} we have $\sum_{j=1}^{d_i}c_{ij}=0$ for each $i$, and therefore $\sum_{j=1}^{d_i}c_{ij}w_{ij}\in I'_E(x,y)$ for each $i$. Since $r$ is minimal, $N=1$. For every two different $i,j\in\{1,\cdots,d_1\}$, since $w_{1i}Rw_{1j}$, $w_{1i}-w_{1j}\in I'_E$. Then $(Q'_E,I'_E)$ satisfies condition $(D)$.

Let $x$, $y$, $z$ be vertices of $Q'_E$ and $u$, $v$, $w$ be paths of $Q'_E$ with $u\in kQ'(x,y)$, $v$, $w\in kQ'(y,z)$, such that $vu$, $wu\notin I'_E$. If $vu+\lambda wu\in I'_E$ for $\lambda\in k^{*}$, by Lemma \ref{a basis of ideal I' E} we have $\lambda=-1$ and $vu$, $wu\in\prescript{}{z}{\mathscr{E}'}_x$ with $(vu)R(wu)$. Let $vu=L(p)$ and $wu=L(q)$ for some standard sequences $p$, $q$ of $E$ with $\prescript{\wedge}{}{p}\equiv \prescript{\wedge}{}{q}$, we may write $p=p_2 p_1$ and $q=q_2 q_1$, where $L(p_1)=L(q_1)=u$, $L(p_2)=v$, $L(q_2)=w$. Then $p_{2}^{\wedge}=p_1 p^{\wedge}\equiv q_1 q^{\wedge}=q_{2}^{\wedge}$. So $\prescript{\wedge}{}{p_2}\equiv \prescript{\wedge}{}{q_2}$ and $vRw$. Therefore $v+\lambda w=v-w\in I'_E$. Then $(Q'_E,I'_E)$ satisfies condition $(C)$.
\end{proof}

Let $\Lambda$ be a locally representation-finite category. For any objects $a$, $b$ of $\Lambda$, $\Lambda(a,b)$ is a uniserial bimodule over $\Lambda(b,b)$ and $\Lambda(a,a)$ (see \cite[Section 2.4]{G}). Denote $$ \mathscr{R}^{0}\Lambda(a,b)=\Lambda(a,b)\supseteq\mathscr{R}^{1}\Lambda(a,b)\supseteq\mathscr{R}^{2}\Lambda(a,b)\supseteq\cdots\supseteq\mathscr{R}^{i}\Lambda(a,b)
\supseteq\cdots $$
the radical series of $\Lambda(a,b)$ as a $\Lambda(b,b)$-$\Lambda(a,a)$-bimodule. A morphism $\mu\in\Lambda(a,b)$ is said to have level $n$ if $\mu\in\mathscr{R}^{n}\Lambda(a,b)-\mathscr{R}^{n+1}\Lambda(a,b)$. We denote by $I=\mathrm{ind}\Lambda$ the category formed by chosen representatives of the isomorphism classes of indecomposable finite-dimensional $\Lambda$-modules, such that $a^{*}=\Lambda(-,a)$ is chosen as a representative for each object $a$ of $\Lambda$. Note that $(-)^{*}$ induces a fully faithful functor $\Lambda\rightarrow I=\mathrm{ind}\Lambda$. Denote
$$\mathrm{rad}^0_{I}=I\supseteq\mathrm{rad}^1_{I}\supseteq\mathrm{rad}^2_{I}\supseteq\cdots\supseteq\mathrm{rad}^p_{I}
\supseteq\cdots$$
the radical series of the category $I$. Since $\mathrm{rad}^{p}_{I}(a^{*},b^{*})$ is a bimodule over $I(b^{*},b^{*})$ and $I(a^{*},a^{*})$, it is equal to some $\mathscr{R}^{n}\Lambda(a,b)^{*}$. Let $g(a,b,n)=\mathrm{sup}\{p\mid \mathrm{rad}^{p}_{I}(a^{*},b^{*})=\mathscr{R}^{n}\Lambda(a,b)^{*}\}$. Define the grade $g(\mu)$ of a morphism $\mu\in\Lambda(a,b)$ of level $n$ by $g(\mu)=g(a,b,n)$. See \cite[Section 1.1]{Br-G}.

\begin{Rem1}\label{codimension}
The number $g(a,b,n)$ is defined for all $a,b\in\Lambda$ and for all non-negative integer $n$: When $n=0$, we have $I(a^{*},b^{*})=\Lambda(a,b)^{*}$, so there exists some $p$ such that $\mathrm{rad}^{p}_{I}(a^{*},b^{*})=\mathscr{R}^{n}\Lambda(a,b)^{*}$. When $n>0$, suppose that $g(a,b,n')$ is defined for all $n'<n$. Denote $g=g(a,b,n-1)$, then $\mathrm{rad}^{g}_{I}(a^{*},b^{*})=\mathscr{R}^{n-1}\Lambda(a,b)^{*}$. We may assume that $g<\infty$, so $\mathrm{rad}^{g+1}_{I}(a^{*},b^{*})\neq\mathscr{R}^{n-1}\Lambda(a,b)^{*}$. According to \cite[Section 2.4]{G}, $\mathscr{R}^{n}\Lambda(a,b)$ has codimension $\leq 1$ in $\mathscr{R}^{n-1}\Lambda(a,b)$. Moreover, \begin{multline*}\mathscr{R}^{n}\Lambda(a,b)^{*}=(\mathscr{R}^{n-1}\Lambda(a,b)\mathrm{rad}\Lambda(a,a)+\mathrm{rad}\Lambda(b,b)\mathscr{R}^{n-1}\Lambda(a,b))^{*} \\=\mathrm{rad}^{g}_{I}(a^{*},b^{*})\mathrm{rad}I(a^{*},a^{*})+\mathrm{rad}I(b^{*},b^{*})\mathrm{rad}^{g}_{I}(a^{*},b^{*})
\\=\mathrm{rad}^{g}_{I}(a^{*},b^{*})\mathrm{rad}_{I}^{1}(a^{*},a^{*})+\mathrm{rad}_{I}^{1}(b^{*},b^{*})\mathrm{rad}^{g}_{I}(a^{*},b^{*})\subseteq\mathrm{rad}^{g+1}_{I}(a^{*},b^{*}). \end{multline*}
Therefore $\mathrm{rad}^{g+1}_{I}(a^{*},b^{*})=\mathscr{R}^{n}\Lambda(a,b)^{*}$. Note that $\mathrm{rad}^{g+1}_{I}(a^{*},b^{*})$ has codimension $1$ in $\mathrm{rad}^{g}_{I}(a^{*},b^{*})$.
\end{Rem1}

Let $\Lambda$ be a locally representation-finite category. According to \cite[Section 2.1]{G}, the universal cover $\widetilde{\Lambda}$ of $\Lambda$ is defined to be the full subcategory of $k(\widetilde{\Gamma}_{\Lambda})$ formed by projective vertices of $\widetilde{\Gamma}_{\Lambda}$, where $\widetilde{\Gamma}_{\Lambda}$ denotes the universal cover of the Auslander-Reiten quiver $\Gamma_{\Lambda}$ of $\Lambda$. Moreover, there exists a covering functor $F:\widetilde{\Lambda}\rightarrow\Lambda$, which is given by the commutative diagram
$$\xymatrix@R=1.3pc{
	k(\widetilde{\Gamma}_{\Lambda})\ar[r]^{E} & I=ind\Lambda \\
    \widetilde{\Lambda}\ar[r]^{F}\ar@{^(-_>}[u]^{incl.} &\Lambda\ar@{^(-_>}[u]_{(-)^{*}}, \\
}$$
where $E:k(\widetilde{\Gamma}_{\Lambda})\rightarrow I$ is a well-behaved functor (\cite[Section 3.1]{BG}).

Since $\widetilde{\Gamma}_{\Lambda}$ is simply connected (that is, it is connected and its fundamental group is trivial), it does not contains oriented cycles. Therefore $k(\widetilde{\Gamma}_{\Lambda})(x,x)\cong k$ for each object $x$ of $k(\widetilde{\Gamma}_{\Lambda})$. Since $\widetilde{\Lambda}$ is a full subcategory of $k(\widetilde{\Gamma}_{\Lambda})$, $\widetilde{\Lambda}(x,x)\cong k$ for each object $x$ of $\widetilde{\Lambda}$. Since $\widetilde{\Gamma}_{\Lambda}$ is a Riedtmann-quiver (\cite[Theorem 2.9]{BG}), by \cite[Proposition 2.4]{BG}, $\widetilde{\Lambda}$ is locally representation-finite. By \cite[Section 1.3]{Br-G}, $\widetilde{\Lambda}$ is schurian, that is, $\mathrm{dim}_{k}\widetilde{\Lambda}(x,y)\leq 1$ for all objects $x,y$ of $\widetilde{\Lambda}$.

Let $M$ be a locally bounded Schurian category. A path
$$ x=x_0\rightarrow x_1\rightarrow \cdots \rightarrow x_p=y $$
of $Q_M$ is said to be nonzero if the composition
$$ M(x_{p-1},x_p)\times\cdots\times M(x_0,x_1)\rightarrow M(x_0,x_p) $$
is nonzero (see \cite[Section 1.3]{Br-G}).

In next two lemmas, we use the following notations from \cite{Br-G}: denote $\overline{\alpha}$ an arrow in $Q_{\Lambda}$ and denote $\alpha$ a representative of $\overline{\alpha}$, that is, a morphism of $\Lambda$ which belongs to $\mathrm{rad}_{\Lambda}(x,y)-\mathrm{rad}^{2}_{\Lambda}(x,y)$ with $x$ (resp. $y$) the source (resp. terminal) of $\overline{\alpha}$ (here $\mathrm{rad}_{\Lambda}$ is the radical of the $k$-category $\Lambda$).

\begin{Lem}\label{stable path} {\rm(\cite[Lemma 1.4]{Br-G})}
Let $\Lambda$ be a locally representation-finite category, $a_0\xrightarrow{\overline{\alpha_{1}}}a_1\xrightarrow{\overline{\alpha_{2}}}\cdots\rightarrow a_{n-1}\xrightarrow{\overline{\alpha_{n}}}a_n$ be a path of $Q_{\Lambda}$ and $\alpha_i\in\mathrm{rad}_{\Lambda}(a_{i-1},a_i)-\mathrm{rad}^{2}_{\Lambda}(a_{i-1},a_i)$ a representative of $\overline{\alpha_{i}}$. The following statements are equivalent: \\ (1) $g(\alpha_n\cdots\alpha_2\alpha_1)=\sum_{i=1}^{n}g(\alpha_i)$; \\ (2) $\alpha'_n\cdots\alpha'_2\alpha'_1\neq 0$ for all representatives $\alpha'_i$ of $\overline{\alpha_{i}}$; \\ (3) $\overline{\alpha_{n}}\cdots\overline{\alpha_{2}}\overline{\alpha_{1}}$ is the projection of a nonzero path of $\widetilde{\Lambda}$, where $\widetilde{\Lambda}$ is the universal cover of $\Lambda$.
\end{Lem}

If a path $u$ of $Q_{\Lambda}$ satisfies one of the statements of Lemma \ref{stable path}, then it is called a stable path. For a stable path $a_0\xrightarrow{\overline{\alpha_{1}}}a_1\xrightarrow{\overline{\alpha_{2}}}\cdots\rightarrow a_{n-1}\xrightarrow{\overline{\alpha_{n}}}a_n$ of $Q_{\Lambda}$, the level of $\alpha_n\cdots\alpha_2\alpha_1$ is independent to the choices of the representatives $\alpha_i$ of $\overline{\alpha_{i}}$. Define the level of the stable path $a_0\xrightarrow{\overline{\alpha_{1}}}a_1\xrightarrow{\overline{\alpha_{2}}}\cdots\rightarrow a_{n-1}\xrightarrow{\overline{\alpha_{n}}}a_n$ to be the level of the morphism $\alpha_n\cdots\alpha_2\alpha_1$.

\begin{Def} \label{definition-stable-contour} {\rm(see \cite[Section 1.5]{Br-G})}
Let $\Lambda$ be a locally representation-finite category. A stable contour of $\Lambda$ is a pair $(v,w)$ of stable paths of $Q_{\Lambda}$ which have the same source, the same terminal and the same level.
\end{Def}

\begin{Lem}\label{stable contour}
Let $\Lambda$ be a locally representation-finite category, $(v,w)$ be a pair of paths of $Q_{\Lambda}$ with the same source $a$ and the same terminal $b$.
\begin{itemize}
  \item[$(1)$] Let $\widetilde{a}$ be a point of the universal cover $\widetilde{\Lambda}$ which lies over $a$, $\widetilde{v}$ be the path of $Q_{\widetilde{\Lambda}}$ with source $\widetilde{a}$ which lies over $v$, and $\widetilde{w}$ be the path of $Q_{\widetilde{\Lambda}}$ with source $\widetilde{a}$ which lies over $w$. Then $(v,w)$ is a stable contour of $\Lambda$ if and only if $\widetilde{v}$ and $\widetilde{w}$ are nonzero paths of $\widetilde{\Lambda}$ with the same terminal.
  \item[$(2)$] Let $\widetilde{b}$ be a point of the universal cover $\widetilde{\Lambda}$ which lies over $b$, $\widetilde{v}$ be the path of $Q_{\widetilde{\Lambda}}$ with terminal $\widetilde{b}$ which lies over $v$, and $\widetilde{w}$ be the path of $Q_{\widetilde{\Lambda}}$ with terminal $\widetilde{b}$ which lies over $w$. Then $(v,w)$ is a stable contour of $\Lambda$ if and only if $\widetilde{v}$ and $\widetilde{w}$ are nonzero paths of $\widetilde{\Lambda}$ with the same source.
\end{itemize}
\end{Lem}

\begin{proof}
We only prove the case that $\widetilde{v}$ and $\widetilde{w}$ have the same source $\widetilde{a}$. Recall that the universal cover $\widetilde{\Lambda}$ of $\Lambda$ is given by the commutative diagram
$$\xymatrix@R=1.3pc{
	k(\widetilde{\Gamma}_{\Lambda})\ar[r]^{E} & I=\mathrm{ind}\Lambda \\
    \widetilde{\Lambda}\ar[r]^{F}\ar@{^(-_>}[u]^{incl.} &\Lambda\ar@{^(-_>}[u]_{(-)^{*}}, \\
}$$
where $\widetilde{\Gamma}_{\Lambda}$ is the universal cover of the Auslander-Reiten quiver $\Gamma_{\Lambda}$ of $\Lambda$ and $E,F$ are covering functors.

"$\Rightarrow$"
Let $v=\overline{\alpha_{n}}\cdots\overline{\alpha_{2}}\overline{\alpha_{1}}$ and $w=\overline{\beta_{m}}\cdots\overline{\beta_{2}}\overline{\beta_{1}}$, where each $\overline{\alpha_{i}}$ and each $\overline{\beta_{j}}$ is an arrow of $Q_{\Lambda}$. Let $\widetilde{v}$ be the path $$ \widetilde{a}=x_0\xrightarrow{\overline{\gamma_1}} x_1\xrightarrow{\overline{\gamma_2}}\cdots x_{n-1}\xrightarrow{\overline{\gamma_n}}x_n$$ and $\widetilde{w}$ be the path $$\widetilde{a}=y_0\xrightarrow{\overline{\delta_1}} y_1\xrightarrow{\overline{\delta_2}}\cdots y_{m-1}\xrightarrow{\overline{\delta_m}}y_m,$$  where each $\overline{\gamma_{i}}$ and each $\overline{\delta_{j}}$ is an arrow of $Q_{\widetilde{\Lambda}}$. For each $1\leq i\leq n$ (resp. $1\leq j\leq m$), choose a representative $\gamma_i\in\mathrm{rad}_{\widetilde{\Lambda}}(x_{i-1},x_i)-\mathrm{rad}^{2}_{\widetilde{\Lambda}}(x_{i-1},x_i)$ of $\overline{\gamma_{i}}$ (resp. a representative $\delta_j\in\mathrm{rad}_{\widetilde{\Lambda}}(y_{j-1},y_i)-\mathrm{rad}^{2}_{\widetilde{\Lambda}}(y_{j-1},y_j)$ of $\overline{\delta_{i}}$), then $\alpha_i=F(\gamma_i)$ (resp. $\beta_j=F(\delta_j)$) is a representative of $\overline{\alpha_i}$ (resp. $\overline{\beta_j}$), where $F:\widetilde{\Lambda}\rightarrow\Lambda$ is the covering functor. Since $(v,w)$ is a stable contour, $g(\alpha_n\cdots\alpha_2\alpha_1)=g(\beta_m\cdots\beta_2\beta_1)$. Let $g=g(\alpha_n\cdots\alpha_2\alpha_1)$.

Since $E:k(\widetilde{\Gamma}_{\Lambda})\rightarrow I$ is a covering functor, it induces an isomorphism $$ \bigoplus_{Ez=b^{*}}\mathrm{rad}^{g}_{k(\widetilde{\Gamma}_{\Lambda})}(\widetilde{a},z)/\mathrm{rad}^{g+1}_{k(\widetilde{\Gamma}_{\Lambda})}(\widetilde{a},z)\rightarrow
\mathrm{rad}^{g}_{I}(a^{*},b^{*})/\mathrm{rad}^{g+1}_{I}(a^{*},b^{*}). $$ By Remark \ref{codimension}, $\mathrm{dim}_{k}(\mathrm{rad}^{g}_{I}(a^{*},b^{*})/\mathrm{rad}^{g+1}_{I}(a^{*},b^{*}))=1$, so there exists a unique object $z$ of $k(\widetilde{\Gamma}_{\Lambda})$ such that $Ez=b^{*}$ and $\mathrm{rad}^{g}_{k(\widetilde{\Gamma}_{\Lambda})}(\widetilde{a},z)/\mathrm{rad}^{g+1}_{k(\widetilde{\Gamma}_{\Lambda})}(\widetilde{a},z)\neq 0$. Since $E(\gamma_n\cdots\gamma_1)=F(\gamma_n\cdots\gamma_1)^{*}=(\alpha_n\cdots\alpha_1)^{*}\in\mathrm{rad}^{g}_{I}(a^{*},b^{*})-\mathrm{rad}^{g+1}_{I}(a^{*},b^{*})$, $\gamma_n\cdots\gamma_1\in\mathrm{rad}^{g}_{k(\widetilde{\Gamma}_{\Lambda})}(\widetilde{a},x_n)-\mathrm{rad}^{g+1}_{k(\widetilde{\Gamma}_{\Lambda})}(\widetilde{a},x_n)$. Therefore $x_n=z$, and $y_m=z$ for the same reason. Moreover, since $\gamma_n\cdots\gamma_1\neq 0$ (resp. $\delta_m\cdots\delta_1\neq 0$), $\widetilde{v}$ (resp. $\widetilde{w}$) is a nonzero path of $Q_{\widetilde{\Lambda}}$.

"$\Leftarrow$"
Since $\widetilde{v}$ and $\widetilde{w}$ are nonzero paths of $\widetilde{\Lambda}$, $v$ and $w$ are stable. Let $\widetilde{v}$ be the path
$$ \widetilde{a}=x_0\xrightarrow{\overline{\gamma_1}} x_1\xrightarrow{\overline{\gamma_2}}\cdots x_{n-1}\xrightarrow{\overline{\gamma_n}}x_n=\widetilde{b}$$
and $\widetilde{w}$ be the path
$$\widetilde{a}=y_0\xrightarrow{\overline{\delta_1}} y_1\xrightarrow{\overline{\delta_2}}\cdots y_{m-1}\xrightarrow{\overline{\delta_m}}y_m=\widetilde{b},$$
where each $\overline{\gamma_{i}}$ and each $\overline{\delta_{j}}$ is an arrow of $Q_{\widetilde{\Lambda}}$. For each $1\leq i\leq n$ (resp. $1\leq j\leq m$), choose a representative $\gamma_i\in\mathrm{rad}_{\widetilde{\Lambda}}(x_{i-1},x_i)-\mathrm{rad}^{2}_{\widetilde{\Lambda}}(x_{i-1},x_i)$ of $\overline{\gamma_{i}}$ (resp. a representative $\delta_j\in\mathrm{rad}_{\widetilde{\Lambda}}(y_{j-1},y_i)-\mathrm{rad}^{2}_{\widetilde{\Lambda}}(y_{j-1},y_j)$ of $\overline{\delta_{i}}$), and let $\alpha_i=F(\gamma_i)$ (resp. $\beta_j=F(\delta_j)$), where $F:\widetilde{\Lambda}\rightarrow\Lambda$ is the covering functor. Since $\widetilde{\Lambda}$ is schurian and $\widetilde{v}$, $\widetilde{w}$ are nonzero paths of $\widetilde{\Lambda}$, there exists some $\lambda\in k^{*}$ such that $\delta_m\cdots\delta_1=\lambda\gamma_n\cdots\gamma_1$. Therefore $\beta_m\cdots\beta_1=\lambda\alpha_n\cdots\alpha_1$ and the stable paths $v$ and $w$ have the same level.

\end{proof}

\begin{Rem1}\label{length of stable contour}
If $(v,w)$ is a stable contour with $v\neq w$, then both $v$ and $w$ are paths of $Q_{\Lambda}$ of length $\geq 2$. The reason is as follows: by Lemma \ref{stable contour}, the pair $(v,w)$ lifts to a pair $(\widetilde{v},\widetilde{w})$ of nonzero paths of $Q_{\widetilde{\Lambda}}$ with same source $\widetilde{a}$ and same terminal $\widetilde{b}$. Suppose $\widetilde{\Lambda}\cong kQ_{\widetilde{\Lambda}}/I_{\widetilde{\Lambda}}$ for some admissible ideal $I_{\widetilde{\Lambda}}$ of $kQ_{\widetilde{\Lambda}}$. Since $\mathrm{dim}_{k}\widetilde{\Lambda}(\widetilde{a},\widetilde{b})=1$ and $\widetilde{v},\widetilde{w}$ are nonzero paths of $\widetilde{\Lambda}$, there exists some $\lambda\in k^{*}$ with $\widetilde{v}-\lambda\widetilde{w}\in I_{\widetilde{\Lambda}}$. Therefore both $\widetilde{v}$ and $\widetilde{w}$ are paths of length $\geq 2$, and the same thing is true for $v$ and $w$.
\end{Rem1}

\begin{Def}{\rm (\cite[Section 3.1]{Br-G})}\label{standard form of l.r.f. category}
The standard form $\overline{\Lambda}$ of a locally representation-finite category $\Lambda$ is the full subcategory of the mesh category $k(\Gamma_{\Lambda})$ of the Auslander-Reiten quiver $\Gamma_{\Lambda}$ of $\Lambda$ formed by projective vertices. Similarly, the standard form $\overline{A}$ of a basic representation-finite algebra $A$ is defined to be the opposite algebra of the algebra $\oplus_{p,q}k(\Gamma_{A})(p,q)$, where $p,q$ range over all projective vertices of $\Gamma_A$ (since we consider left $A$-modules, we use opposite algebra in the definition of the standard form of $A$).
\end{Def}

It can be shown that for a locally representation-finite category $\Lambda$ (resp. a basic representation-finite algebra $A$), $\overline{\Lambda}$ (resp. $\overline{A}$) is standard and the Auslander-Reiten quivers of $\Lambda$ and $\overline{\Lambda}$ (resp. $A$ and $\overline{A}$) are isomorphic (see \cite{BG}). For a locally representation-finite category $\Lambda$, the standard form $\overline{\Lambda}$ of $\Lambda$ is described as follows.

\begin{Prop}{\rm (\cite[Theorem 3.1]{Br-G})}\label{construction-of-standard-form}
Let $\Lambda$ be a locally representation-finite category. For any two objects $x$, $y$ of $\Lambda$, let $I_{\Lambda}(x,y)$ be the subspace of $kQ_{\Lambda}(x,y)$ which is generated by the non-stable paths and by the differences $v-w$, where $(v,w)$ ranges over the stable contours with source $x$ and terminal $y$. Then $I_{\Lambda}$ is an ideal of the path category $kQ_{\Lambda}$ and the standard form $\overline{\Lambda}$ of $\Lambda$ is isomorphic to $kQ_{\Lambda}/I_{\Lambda}$.
\end{Prop}

We abbreviate indecomposable, basic, representation-finite self-injective algebra over $k$ (not isomorphic to the underlying field $k$) by RFS algebra.

\begin{Lem}\label{socle paths become stable contour}
Let $A$ be an RFS algebra, $\overline{A}=kQ_{A}/I_{A}$ be the standard form of $A$, $v$, $w$ be two paths of $Q_{A}$ with the same source (resp. terminal) which belong to $\mathrm{soc}\overline{A}-\{0\}$. Then $(v,w)$ is a stable contour of $A$.
\end{Lem}

\begin{proof}
Since $\overline{A}$ is self-injective and $v$, $w\in\mathrm{soc}\overline{A}-\{0\}$, they also have the same terminal. Since $v$ and $w$ are nonzero in $\overline{A}$, they are stable paths. Since $v$ and $w$ belong to $\mathrm{soc}\overline{A}-\{0\}$, which have the same source, there exists some $\lambda\in k^{*}$ such that $w=\lambda v$ in $\overline{A}$. By the definition of $I_{A}$, we imply that $\lambda=1$ and $(v,w)$ is a stable contour of $A$.
\end{proof}

\begin{Lem}\label{socle paths}
Let $A$ be an RFS algebra of Loewy length $>2$, $\overline{A}=kQ_{A}/I_{A}$ be the standard form of $A$. Let $\alpha_n\cdots\alpha_2\alpha_1$ be a path of $Q_{A}$ which belongs to $\mathrm{soc}\overline{A}-\{0\}$, $p$ (resp. $q$) be a path of $Q_{A}$ such that $p\alpha_n\cdots\alpha_2$ (resp. $\alpha_{n-1}\cdots\alpha_1 q$) belongs to $\mathrm{soc}\overline{A}-\{0\}$, then $l(p)=1$ (resp. $l(q)=1$) and $p$ (resp. $q$) is unique.
\end{Lem}

\begin{proof}
Since $\alpha_n\cdots\alpha_2\notin\mathrm{soc}\overline{A}$, $l(p)\geq 1$. Suppose $l(p)\geq 2$, write $p=\alpha_{n+l}\cdots\alpha_{n+1}$ where $l\geq 2$. Since $\alpha_{n+1}\alpha_n\cdots\alpha_2\notin\mathrm{soc}\overline{A}$, there exists arrows $\alpha'_1$, $\alpha'_0$, $\cdots$, $\alpha'_{2-k}$ ($k\geq 1$) of $Q_{A}$ such that $\alpha_{n+1}\alpha_n\cdots\alpha_2\alpha'_1\cdots\alpha'_{2-k}\in\mathrm{soc}\overline{A}-\{0\}$. Similarly, since $\alpha_n\cdots\alpha_2\alpha'_1\cdots\alpha'_{2-k}\notin\mathrm{soc}\overline{A}$, there exists arrows $\alpha'_{1-k}$, $\cdots$, $\alpha'_{2-k-l}$ ($l\geq 1$) of $Q_{A}$ such that $\alpha_n\cdots\alpha_2\alpha'_1\cdots\alpha'_{2-k-l}\in\mathrm{soc}\overline{A}-\{0\}$. According to Lemma \ref{socle paths become stable contour}, the pair of paths $(\alpha_n\cdots\alpha_2\alpha_1,\alpha_n\cdots\alpha_2\alpha'_1\cdots\alpha'_{2-k-l})$ is a stable contour of $A$. By Lemma \ref{stable contour}, $(\alpha_1,\alpha'_1\cdots\alpha'_{2-k-l})$ is also a stable contour of $A$. Since $l(\alpha'_1\cdots\alpha'_{2-k-l})\geq 2$, $\alpha_1\neq\alpha'_1\cdots\alpha'_{2-k-l}$. But it contradicts to Remark \ref{length of stable contour}.

To show $p$ is unique, let $\alpha_{n+1}$, $\alpha'_{n+1}$ be two arrows of $Q_{A}$ such that both $\alpha_{n+1}\alpha_n\cdots\alpha_2$ and $\alpha'_{n+1}\alpha_n\cdots\alpha_2$ belong to $\mathrm{soc}\overline{A}-\{0\}$. Similarly, we can show that $(\alpha_{n+1},\alpha'_{n+1})$ is a stable contour of $A$. Then by Remark \ref{length of stable contour} we have $\alpha_{n+1}=\alpha'_{n+1}$.
\end{proof}

Let $A$ be an RFS algebra of Loewy length $>2$ with the standard form $\overline{A}=kQ_{A}/I_{A}$. For each path $v=\alpha_n\cdots\alpha_2\alpha_1$ of $Q_{A}$ which belongs to $\mathrm{soc}\overline{A}-\{0\}$, let $\alpha_{n+1}$ (resp. $\alpha_0$) be the unique arrow of $Q_{A}$ such that $\alpha_{n+1}\alpha_n\cdots\alpha_2\in\mathrm{soc}\overline{A}-\{0\}$ (resp. $\alpha_{n-1}\cdots\alpha_1\alpha_0\in\mathrm{soc}\overline{A}-\{0\}$). Define $v[1]=\alpha_{n+1}\alpha_n\cdots\alpha_2$ and $v[-1]=\alpha_{n-1}\cdots\alpha_1\alpha_0$. Define $v[m]=(v[m-1])[1]$ and $v[-m]=(v[1-m])[-1]$ for each $m>1$ inductively. Let $$E=\{e_v\mid v\mbox{ is a path of }Q_{A}\mbox{ which belongs to }\mathrm{soc}\overline{A}-\{0\}\}$$ be a set, define the $G$-set structure on $E$ by setting $g^{m}\cdot e_v=e_{v[m]}$ for each $e_v\in E$ and for each $m\in\mathbb{Z}$, where $G=\langle g\rangle$ is an infinite cyclic group. For each path $v$ of $Q_{A}$ which belongs to $\mathrm{soc}\overline{A}-\{0\}$, let $\alpha_v$ be the initial arrow of $v$. Define two partitions $P,L$ on $E$ as follows: for every $e_v\in E$, $P(e_v)=\{e_w\mid s(v)=s(w)\}$ and $L(e_v)=\{e_w\mid \alpha_v=\alpha_w\}$, where $s(p)$ denotes the source of the path $p$. Define a function $d:E\rightarrow\mathbb{Z}_{+}$ by setting $d(e_v)$ to be the length $l(v)$ of $v$.

\begin{Prop}\label{A-is-a-f-s-BCA}
Let $A$ be an RFS algebra of Loewy length $>2$ with the standard form $\overline{A}=kQ_{A}/I_{A}$ and let $E=(E,P,L,d)$ be defined as above. Then $E$ is a finite $f_s$-BC and the opposite algebra $(A_E)^{op}$ of the corresponding $f_s$-BCA $A_{E}$ is isomorphic to $\overline{A}$.
\end{Prop}

\begin{proof}
{\it Step 1: To show that $E=(E,P,L,d)$ is a finite $f_s$-BC.}

$(f1)$ Since $I_{A}$ is an admissible ideal and $Q_{A}$ is finite, the number of nonzero paths of $Q_{A}$ in $A$ is finite. Therefore $E$ and each class $P(e_v)$ are finite sets. By definition we have $L(e_v)\subseteq P(e_v)$ for any $e_v\in E$.

$(f2)$ If $L(e_v)=L(e_w)$, then $\alpha_v=\alpha_w$, therefore $s(v[1])=t(\alpha_v)=t(\alpha_w)=s(w[1])$ and $P(g\cdot e_v)=P(e_{v[1]})=P(e_{w[1]})=P(g\cdot e_w)$.

$(f3)$ Since $l(v)=l(v[n])$ for any path $v$ which belongs to $\mathrm{soc}\overline{A}-\{0\}$ and for any $n\in\mathbb{Z}$, $d(e_v)=d(e_w)$ for any $e_v$, $e_w\in E$ which belong to the same $\langle g\rangle$-orbit.

$(f4)$ If $P(e_v)=P(e_w)$, then $s(v)=s(w)$. Since $A$ is self-injective and $v$, $w\in\mathrm{soc}\overline{A}-\{0\}$, $v$ and $w$ have the same terminal. Therefore $v[l(v)]$ and $w[l(w)]$ have the same source. Since $\sigma(e_v)=e_{v[l(v)]}$ (resp. $\sigma(e_w)=e_{w[l(w)]}$), $P(\sigma(e_v))=P(\sigma(e_w))$. Similarly we have $P(\sigma^{-1}(e_v))=P(\sigma^{-1}(e_w))$.

$(f5)$ If $L(e_v)=L(e_w)$, then $\alpha_v=\alpha_w$. Let $v=v'\alpha_v$ and $w=w'\alpha_w$. By Lemma \ref{socle paths become stable contour}, $(v,w)$ is a stable contour of $A$, and by Lemma \ref{stable contour}, $(v',w')$ is also a stable contour of $A$. Let $\beta$ be the arrow of $Q_{A}$ such that $\beta v'\in\mathrm{soc}\overline{A}-\{0\}$. Since $(v',w')$ is a stable contour of $A$, $v'-w'\in I_{A}$. Therefore $v'=w'$ in $\overline{A}$, and $\beta w'=\beta v'\in\mathrm{soc}\overline{A}-\{0\}$. Thus $v[1]=\beta v'$ and $w[1]=\beta w'$. Since $\alpha_{v[l(v)]}=\beta=\alpha_{w[l(w)]}$, $L(\sigma(e_v))=L(\sigma(e_w))$. Similarly, we have $L(\sigma^{-1}(e_v))=L(\sigma^{-1}(e_w))$.

$(f6)$ For $e_v$, $e_w\in E$, if $d(e_v)\leq d(e_w)$ and $L(e_{v[i]})=L(e_{w[i]})$ for each $0\leq i<d(e_v)$, then $w=w'v$ for some path $w'$ of $Q_{A}$. Since $v$, $w\in\mathrm{soc}\overline{A}-\{0\}$, we have $w=v$.

$(f7)$ Note that each standard sequence $(g^{m-1}\cdot e_v,\cdots,g\cdot e_v,e_v)$ (where $v=\alpha_n\cdots\alpha_2\alpha_1$ and $0\leq m\leq l(v)$) of $E$ can be regarded as a pair of paths $(v'',v')$ of $Q_{A}$ such that $v'=\alpha_m\cdots\alpha_2\alpha_1$, $v''=\alpha_n\cdots\alpha_{m+2}\alpha_{m+1}$, $v''v'=v$. Conversely, each pair of paths $(v'',v')$ of $Q_{A}$ with $v''v'\in\mathrm{soc}\overline{A}-\{0\}$ corresponds to a standard sequence $(g^{m-1}\cdot e_v,\cdots,g\cdot e_v,e_v)$ of $E$, where $v=v''v'$ and $m=l(v')$. Therefore we may consider a standard sequence of $E$ as such a pair. Two standard sequences $(v'',v')$ and $(w'',w')$ of $E$ are identical if and only if $v'=w'$. Therefore for each standard sequence $(v'',v')$ of $E$, we have $[\prescript{\wedge}{}{(v'',v')}]=\{(w,v'')\mid wv''\in\mathrm{soc}\overline{A}-\{0\}\}$. For identical standard sequences $(u,v)$, $(w,v)$ of $E$, we have $[\prescript{\wedge}{}{(u,v)}]^{\wedge}=\{(u,v')\mid uv'\in\mathrm{soc}\overline{A}-\{0\}\}$ and $[[\prescript{\wedge}{}{(u,v)}]^{\wedge}]=\{(u',v')\mid uv',u'v'\in\mathrm{soc}\overline{A}-\{0\}\}$. Also we have $[[\prescript{\wedge}{}{(w,v)}]^{\wedge}]=\{(w',v')\mid wv',w'v'\in\mathrm{soc}\overline{A}-\{0\}\}$. By Lemma \ref{socle paths become stable contour}, $(uv,wv)$ is a stable contour of $A$, then by Lemma \ref{stable contour} $(u,w)$ is also a stable contour of $A$. Therefore $u=w$ in $\overline{A}$, and $uv'\in\mathrm{soc}\overline{A}-\{0\}$ if and only if $wv'\in\mathrm{soc}\overline{A}-\{0\}$. Thus $[[\prescript{\wedge}{}{(u,v)}]^{\wedge}]=[[\prescript{\wedge}{}{(w,v)}]^{\wedge}]$.

\medskip
{\it Step 2: To show $\overline{A}\cong (A_{E})^{op}$.}

Define a quiver morphism $f:Q_E\rightarrow Q_{A}$ which maps each vertex $P(e_v)$ to $s(v)$ and maps each arrow $L(e_v)$ to $\alpha_v$. It is straightforward to show that $f$ is a quiver isomorphism. To show $f$ induces an isomorphism between $(A_{E})^{op}=kQ_E/I_E$ and $\overline{A}$, it suffices to show that $\widetilde{f}(I_E)=I_{A}$, where $\widetilde{f}$ is the isomorphism $kQ_E\rightarrow kQ_{A}$ of path categories induced by $f$.

If $L(e_{v[m-1-k]})\cdots L(e_{v[1]})L(e_v)-L(e_{w[m'-1-k]})\cdots L(e_{w[1]})L(e_w)$ is a relation of $I_E$ of type $(fR1)$, where $s(v)=s(w)$, $m=l(v)$, $m'=l(w)$, and $L(e_{v[m-i]})=L(e_{w[m'-i]})$ for $1\leq i\leq k$, then we may assume that $v=pq$ and $w=pq'$, where $p$ is a path of length $k$. By Lemma \ref{socle paths become stable contour} and Lemma \ref{stable contour}, $(q,q')$ is a stable contour of $A$. Therefore $\widetilde{f}(L(e_{v[m-1-k]})\cdots L(e_{v[1]})L(e_v)-L(e_{w[m'-1-k]})\cdots L(e_{w[1]})L(e_w))=q-q'\in I_{A}$.

If $L(e_{v_n})\cdots L(e_{v_2})L(e_{v_1})$ is a relation of $I_E$ of type $(fR2)$, then the path \\ $p=f(L(e_{v_n})\cdots L(e_{v_2})L(e_{v_1}))=\alpha_{v_n}\cdots\alpha_{v_2}\alpha_{v_1}$ of $Q_{A}$ is not a subpath of any path $v$ of $Q_{A}$ which belongs to $\mathrm{soc}\overline{A}-\{0\}$. Therefore $p$ is non-stable and it belongs to $I_{A}$.

If $L(e_{v[n]})\cdots L(e_{v[1]})L(e_v)$ is a relation of $I_E$ of type $(fR3)$, where $n\geq l(v)$, then \\ $f(L(e_{v[n]})\cdots L(e_{v[1]})L(e_v))$ is a path of $Q_{A}$ which contains $v$ as a proper subpath. Therefore $f(L(e_{v[n]})\cdots L(e_{v[1]})L(e_v))=0$ in $\overline{A}$. So we have $\widetilde{f}(I_E)\subseteq I_{A}$.

For any non-stable path $p$ of $A$, we have $p=0$ in $\overline{A}$. Then $p$ is not a subpath of any path of $Q_A$ which belongs to $\mathrm{soc}\overline{A}-\{0\}$. Let $q$ be the preimage of $p$ under $\widetilde{f}$. If $q$ is not a relation of $I_E$ of type $(fR2)$, then $q=L(e_{v[n-1]})\cdots L(e_{v[1]})L(e_v)$ for some $e_v\in E$ and some $n>0$. Since $p$ is not a subpath of $v$, $n>l(v)$. Therefore $q$ is a relation of $I_E$ of type $(fR3)$.

For any stable contour $(v,w)$ of $A$, since $v\neq 0$ in $\overline{A}$, there exists some path $u$ of $Q_{A}$ such that $uv\in\mathrm{soc}\overline{A}-\{0\}$. Since $v=w$ in $\overline{A}$, $wv\in\mathrm{soc}\overline{A}-\{0\}$. Let $p=uv$, $q=uw$. Then $L(e_{p[m-k-1]})\cdots L(e_{p[1]})L(e_p)-L(e_{q[m'-k-1]})\cdots L(e_{q[1]})L(e_q)$ is a relation of $I_E$ of type $(fR1)$, where $m=l(p)$, $m'=l(q)$, $k=l(u)$, which is the preimage of $v-w$ under $\widetilde{f}$. Therefore the ideal $I_A$ of $Q_A$ is contained in $\widetilde{f}(I_E)$.
\end{proof}

\begin{Thm}\label{RFS-algebra=r.f.fsBCA}
The class of standard RFS algebras is equal to the class of finite-dimensional indecomposable representation-finite $f_s$-BCAs.
\end{Thm}

\begin{proof}
Let $A$ be a standard RFS algebra. If the Loewy length of $A$ is larger than $2$, then $B:=A^{op}$ is also a standard RFS algebra whose Loewy length is larger than $2$. By Proposition \ref{A-is-a-f-s-BCA}, there exists an $f_s$-BC $E$ with $A_E\cong\overline{B}^{op}$. Since $B$ is standard, the standard form $\overline{B}$ is isomorphic to $B$. Then $A_E\cong B^{op}\cong A$, and $A$ is an $f_s$-BCA. If the Loewy length of $A$ is $2$, then $A$ is given by the quiver
$$
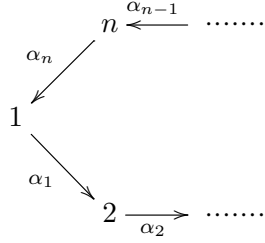
\vcenter{
	\xymatrix@R=1.9pc@C=1.9pc{
	& n \ar[dl]_{\alpha_n} & \ar[l]_{\alpha_{n-1}} \ar@{{}*{\cdot}{}}[r] &  \\
	1  \ar[dr]_{\alpha_1} & & & \\
	& 2 \ar[r]_{\alpha_2} & \ar@{{}*{\cdot}{}}[r]&\\
}}$$
with relations $\alpha_{i+1}\alpha_i=0$, where $i\in\{0,1,\cdots,n-1\}=\mathbb{Z}/n\mathbb{Z}$. Define an $f_s$-BC $E=(E,P,L,d)$ as follows: $E=\{e_0,e_1,\cdots,e_{n-1}\}$ with the $G$-set structure given by $g\cdot e_i=e_{i+1}$, where $i\in\{0,1,\cdots,n-1\}=\mathbb{Z}/n\mathbb{Z}$; the partitions $P$ and $L$ are given by $P(e)=L(e)=\{e\}$ for every $e\in E$; the degree function $d$ is given by $d(e)=1$ for every $e\in E$. Then $A$ is isomorphic to $A_E$, which implies that $A$ is an $f_s$-BCA.

Conversely, let $A_E$ be a finite-dimensional indecomposable representation-finite $f_s$-BCA. By Proposition \ref{l.r.f. fsBCC is standard} and Proposition \ref{fsBCA-is-Frobenius}, $A_E$ is standard and self-injective. Moreover, by Proposition \ref{a basis of J} and Theorem \ref{f-BC algebra is locally bounded}, the radical of $A$ is nonzero, so $A$ is not isomorphic to $k$. Therefore $A_E$ is a standard RFS algebra.
\end{proof}

\begin{Cor}
The class of finite-dimensional representation-finite $f_s$-BCAs is closed under derived equivalence.
\end{Cor}

\begin{proof}
It follows from Theorem \ref{RFS-algebra=r.f.fsBCA} and the fact that the class of standard RFS algebras is closed under derived equivalence (by \cite{A} and \cite{AlR}).
\end{proof}

We illustrate the discussion in this section with an example.

\begin{Ex1}
Let $A=kQ/I$, where $Q$ is the quiver

$$
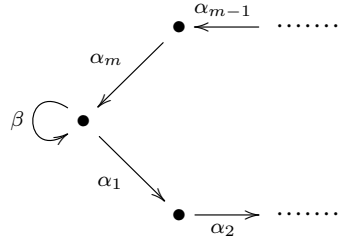
\vcenter{
\xymatrix {
	& \bullet \ar[dl]_{\alpha_m} & \ar[l]_{\alpha_{m-1}} \ar@{{}*{\cdot}{}}[r] &  \\
	\bullet \ar@(ul,dl)_{\beta} \ar[dr]_{\alpha_1} & & & \\
	& \bullet \ar[r]_{\alpha_2} & \ar@{{}*{\cdot}{}}[r]&\\
}}$$
$(m\geq 2)$, and $I$ is the ideal of $kQ$ generated by $\alpha_m\cdots\alpha_2\alpha_1-\beta^2$, $\alpha_1\alpha_m$, $\alpha_i\cdots\alpha_1\beta\alpha_m\cdots\alpha_i$ $(1\leq i\leq m)$. $A$ is a standard RFS algebra of type $(D_{3m},\frac{1}{3},1)$. Let $\overline{A}=kQ/I_A$ be the standard form of $A$, where $I_A$ is the ideal of $kQ$ generated by non-stable paths and by differences $v-w$ with $(v,w)$ being a stable contour of $A$. It is straightforward to show that $(\alpha_m\cdots\alpha_2\alpha_1,\beta^2)$ is a stable contour of $A$ and the paths $\alpha_1\alpha_m$, $\alpha_i\cdots\alpha_1\beta\alpha_m\cdots\alpha_i$ $(1\leq i\leq m)$ are non-stable paths. So the ideal $I$ is contained in $I_A$. Since $A$ is standard, $\overline{A}$ is isomorphic to $A$, and $\mathrm{dim}_{k}A=\mathrm{dim}_{k}\overline{A}$. Therefore $I=I_A$.

We now construct the $f_s$-BC $E=(E,P,L,d)$ associated to $\overline{A}$ (see the paragraph before Proposition \ref{A-is-a-f-s-BCA}). The set of paths of $Q$ which belong to soc$\overline{A}-\{0\}$ is $$\{\beta\alpha_m\cdots\alpha_1,\beta^{3},\alpha_m\cdots\alpha_1\beta,\alpha_1\beta\alpha_m\cdots\alpha_2,\alpha_2\alpha_1\beta\alpha_m\cdots\alpha_3,
\cdots,\alpha_{m-1}\cdots\alpha_1\beta\alpha_m\}.$$ Denote $1=e_{\alpha_m\cdots\alpha_1\beta}$, $1'=e_{\beta\alpha_m\cdots\alpha_1}$, $1''=e_{\beta^{3}}$, and $i=e_{\alpha_{i-1}\cdots\alpha_1\beta\alpha_m\cdots\alpha_i}$ for $2\leq i\leq m$, then $E=\{1,1',1'',2,\cdots,m\}$ with the $G$-set structure given by $g\cdot 1=1'$, $g\cdot 1'=2$, $g\cdot 2=3$, $\cdots$, $g\cdot (m-1)=m$, $g\cdot m=1$, $g\cdot 1''=1''$. The partition $P$ is given by $P(1)=\{1,1',1''\}$ and $P(e)=\{e\}$ for other $e\in E$. The partition $L$ is given by $L(1)=\{1,1''\}$ and $L(e)=\{e\}$ for other $e\in E$. The degree function $d$ is given by $d(1'')=3$ and $d(e)=m+1$ for other $e\in E$. $E$ can be visualized by the following diagram, where $(3)$ and $(m+1)$ denote the degrees and the small arcs at the angles $1$ and $1''$ mean that these two angles belong to the same class of the partition $L$.

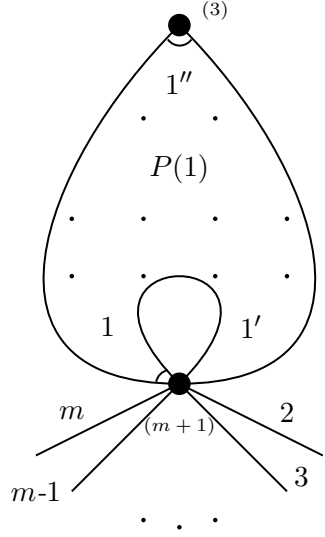
\begin{center}
\tikzset{every picture/.style={line width=0.75pt}}
\begin{tikzpicture}[x=15pt,y=15pt,yscale=1.8,xscale=1.8]
\fill (0,5) circle (0.5ex);
\fill (0,0) circle (0.5ex);
\draw    (0,0) -- (2,-1) ;
\draw    (0,0) -- (1.5,-1.5) ;
\draw    (0,0) -- (-2,-1) ;
\draw    (0,0) -- (-1.5,-1.5) ;
\draw    (0,0) .. controls (-3,0) and (-2,3) .. (0,5) ;
\draw    (0,0) .. controls (3,0) and (2,3) .. (0,5) ;
\draw    (0,0) .. controls (-2,2) and (2,2) .. (0,0) ;
\node at(0,4.2) {$1''$}; \draw (-0.16,4.8) arc (215:325:0.2);
\node at(0.5,5.2) {\tiny$(3)$}; \node at(0,-0.6) {\tiny$(m+1)$};
\node at(0,3) {$P(1)$};
\node at(-1,0.8) {$1$}; \draw (-0.16,0.19) arc (100:170:0.2);
\node at(1,0.8) {$1'$};
\node at(1.5,-0.4) {$2$};
\node at(1.7,-1.3) {$3$};
\node at(-1.5,-0.4) {$m$};
\node at(-2,-1.5) {$m$-$1$};
\fill (-0.5,3.7) circle (0.1ex);
\fill (0.5,3.7) circle (0.1ex);
\fill (-1.5,2.3) circle (0.1ex);
\fill (-0.5,2.3) circle (0.1ex);
\fill (0.5,2.3) circle (0.1ex);
\fill (1.5,2.3) circle (0.1ex);
\fill (-1.5,1.5) circle (0.1ex);
\fill (-0.5,1.5) circle (0.1ex);
\fill (0.5,1.5) circle (0.1ex);
\fill (1.5,1.5) circle (0.1ex);
\fill (0,-2) circle (0.1ex);
\fill (-0.5,-1.9) circle (0.1ex);
\fill (0.5,-1.9) circle (0.1ex);
\end{tikzpicture}
\end{center}

\medskip
Let $A_E$ be the associated f-BCA. Then we have the following structure of the indecomposable projective $A_{E}^{op}$-modules:

\medskip
$P_1=$ \xymatrix@R=0.6pc@C=0.8pc {&1\ar@{-}[dl]\ar@{-}[dr]& \\
1 \ar@{-}[d]\ar@{-}[ddddrr]& & 2\ar@{-}[d] \\
2 \ar@{.}[dd] &  & 3 \ar@{.}[dd] \\
&& \\
m-1\ar@{-}[d] & & m\ar@{-}[d] \\
m \ar@{-}[dr]& & 1\ar@{-}[dl] \\
& 1 & }, $P_2=$ \xymatrix@R=0.6pc@C=0.8pc { 2 \ar@{-}[d]\\
 3 \ar@{.}[dd]\\
 \\
 m \ar@{-}[d]\\
 1 \ar@{-}[d]\\
 1\ar@{-}[d] \\
 2}, $\dots$, $P_m=$ \xymatrix@R=0.6pc@C=0.8pc { m \ar@{-}[d]\\
 1 \ar@{-}[d]\\
 1\ar@{-}[d] \\
 2 \ar@{-}[d]\\
3 \ar@{.}[dd] \\
\\
 m}.\\

Note that there is a standard sequence $p=(1'',1'')$ associated to the vertex $G\cdot 1''$ and there is a standard sequence $q=(m,\cdots,3,2,1')$ associated to the vertex $G\cdot 1$. Since $L(1)=\{1,1''\}$, the left complements $\prescript{\wedge}{}{p}$ and $\prescript{\wedge}{}{q}$ are identical, we have a relation $L(p)-L(q)$ of type $(fR1)$, which corresponds
to the relation $\beta^{2}-\alpha_m\cdots\alpha_1$ in $\overline{A}$.
\end{Ex1}

\end{document}